\documentclass[a4paper,12pt,reqno]{article}

 \usepackage[english]{babel}
 \usepackage{amsmath,amsfonts,amssymb,amsthm}
 \usepackage{graphicx} 
 \usepackage{subcaption}

\usepackage{enumitem} %per fare il "-" su itemize  
   	
   	\usepackage{hyperref}
	\usepackage{xcolor}
	\hypersetup{
    	colorlinks,
    	linkcolor={red!80!black},
    	citecolor={blue!70!black},
    	urlcolor={blue!90!black}
		}
    \usepackage{authblk}

\numberwithin{equation}{section}

\newtheorem{theorem}{Theorem}[section]
\newtheorem{corollary}[theorem]{Corollary}
\newtheorem{lemma}[theorem]{Lemma}
\newtheorem{proposition}[theorem]{Proposition}

\theoremstyle{definition}
\newtheorem{definition}[theorem]{Definition}

\renewcommand{\leq}{\leqslant}

\renewcommand{\geq}{\geqslant}

\def\N{\mathbb{N}}
\def\Z{\mathbb{Z}}

\def\R{\mathbb{R}}

\newcommand{\norm}[1]{\left\lVert#1\right\rVert}

\title{A Fisher-KPP model with \\ a fast
	diffusion line in periodic media\thanks{It is a pleasure
		to thank Luca Rossi and Henri Berestycki
		for very interesting discussions.}}

\author{Elisa Affili\thanks{Dipartimento di Matematica, Universit\`a degli studi di Milano,
		Via Saldini 50, 20133 Milan, Italy, and 
		Centre d'Analyse et de Math\'ematique Sociales,
		\'Ecole des Hautes \'Etudes en Sciences Sociales,
		54 Boulevard Raspail,
		75006 Paris, France. {\tt elisa.affili@unimi.it}}}

\begin{document}

\maketitle

\begin{abstract}
We treat a model of population dynamics in a periodic environment presenting a fast diffusion line. The
``road-field'' model, introduced in \cite{brr}, is a system of coupled reaction-diffusion equations set in domains of different dimensions. Here, we consider for the first time the case of a reaction term depending on a spatial variable in a periodic fashion, which is of great interest for both its mathematical difficulties and for its applications. We derive necessary and sufficient conditions for the survival of the species in terms of the sign of a suitable generalised principal eigenvalue, defined recently in \cite{romain}. Moreover, we compare the long time behaviour of a population in the same environment without the fast diffusion line, finding  that this element has no impact on the survival chances.
\end{abstract}


\vspace{0.5em}


\section{Setting and main results}

This paper investigates some effects of a fast diffusion line in an ecological dynamics problem.
Various examples in the literature showed that, in the presence of roads or trade lines, some species or infections spread faster along these lines, and then diffuse in the surroundings. This was observed in the case of the Processionary caterpillar, whose spreading in France and Europe has been accelerated by accidental human transport \cite{robinet2012human}. Another striking proof was given in  \cite{gatto2020spread}, where the authors point out that the COVID-19 epidemics in Northern Italy at the beginning of 2020 diffused faster along the highways.  

A model for biological diffusion in a homogeneous medium presenting a fast diffusion line was proposed by Berestycki, Roquejoffre and Rossi in \cite{brr}, and since then is called the \emph{road-field model}. The authors proved an acceleration effect due to the road on the spreading speed of an invading species. 
Since then, a growing number of articles treated variations of the same system, investigating in particular the effect of different type of diffusion or different geometries \cite{berestycki2014speed, berestycki2015effect, rossi2017effect}.

However, natural environments are usually far from being homogeneous and, more often than not, territories are a composition of different habitats. Living conditions and heterogeneity play a strong impact on the survival chances of a species and on the equilibria at which the population can settle. 

Road-field models on heterogeneous environments have been  little studied so far, being more complex to treat. One of the few example is the paper \cite{giletti2015kpp} for periodic exchange terms between the population on the road and the one in the field. Recently, Berestycki, Ducasse and Rossi introduced a notion of generalised principal eigenvalue for the road-field system in \cite{romain} and, thanks to it, they were able to treat the case of an ecological niche facing climate change in \cite{econiches}. 
%This last framework prescribes a limited favourable zone enclosed in an unfavourable environment; if there is climate change, the favourable niche moves in time in a precise direction.

Here, we propose an analysis of the asymptotic behaviour of an invasive population under the assumption of spatial periodicity of the reaction term. Of course, under this hypothesis we can investigate deeper the dependence of the population on a natural-like environment and the effects of the road in this balance.
Under which conditions does the population survive in a periodic medium? And does the road play some role on the survival chances of a species, perturbing the environment and scattering the individuals, or rather permitting them to reach advantageous zones more easily?
These are the questions we are going to tackle.

\subsection{The model}

In this paper, we study the reaction-diffusion model regulating the dynamics of a population living in a periodic environment with a fast diffusion channel. The equivalent of this model for homogeneous media was first introduced by Berestycki, Roquejoffre and Rossi in \cite{brr}. Consider the half plane $\Omega:=\R\times \R^+$, where we mean $\R^+=(0, +\infty)$.
The proposed model imposes the diffusion of a species in $\Omega$ and prescribes that on $\partial \Omega=\R\times \{ y=0\}$ the population diffuses at a different speed. 
We call $v(x,t)$ the density of population for $(x,y)\in\Omega$, hence on the ``field'', and $u(x)$ the density of population for $x\in\R$, i.e. on the ``road''; moreover, we take $D$, $d$, $\nu$, $\mu$ positive constants and $c\geq 0$. Then, the system we analyse reads
\begin{equation}\label{ch1sys:fieldroad}
\left\{
\begin{array}{lr}
\partial_t u-D  u '' -c u' - \nu  v|_{y=0} + \mu u= 0,   & x\in \R,  \\
\partial_t v -d \Delta v-c\partial_x v  =f(x,v),  & (x, y)\in \Omega, \\
-d  \partial_y{v}|_{y=0} + \nu v|_{y=0} -\mu u=0, & x\in\R.
\end{array}
\right.
\end{equation}
In $\Omega$, the population evolves with a net birth-death rate represented by $f$, that depends on the variable $x$. 
This embodies the heterogeneity of the media: in fact, environments are typically not uniform and some zone are more favourable than others. 
There is no dependence in the variable $y$, since the presence of the road itself creates enough heterogeneity in that direction.
The function  $f:\R\times \R_{\geq 0}\to \R $
is always supposed to be $\mathcal{C}^{1}$ in $x$, locally in $v$, and Lipschitz in $v$, uniformly in $x$; moreover, we suppose that the value $v=0$ is an equilibrium, that is
\begin{equation}\label{ch1hyp:0}
f(x,0)=0, \quad \text{for all} \ x\in \R,
\end{equation}
and that 		
\begin{equation}\label{ch1hyp:M}
\exists M>0 \ \text{such that} \ f(x, v)<0 \quad \text{for all} \ v>M \ \text{and all} \ x\in \R. 
\end{equation}
We will derive some inequalities on the generalised principal eigenvalue of \eqref{ch1sys:fieldroad} for the general case of $f$ respecting these hypothesis and 
$c$ possibly nonzero.

The characterisation of extinction or persistence of the species is performed for the case of $c=0$ and $f$ a periodic function, reflecting the periodicity of the environment in which the population diffuses. 
We will analyse the  case of a KPP nonlinearity, that is, we require that 
\begin{equation}\label{ch1hyp:KPP}
\frac{f(x,s_2)}{s_2}< \frac{f(x,s_1)}{s_1} \quad  \text{for all} \ s_2>s_1>0 \  \text{and all} \ x\in\R.
\end{equation}
Then, we suppose that there exists $\ell> 0$ such that
\begin{equation}\label{ch1hyp:per}
f(x+\ell, s)=f(x,s) \quad \text{for all} \ s >0 \  \text{and all} \ x\in\R.
\end{equation}

%Modello senza strada

To study the effect of the line of fast diffusion, we will compare the behaviour of \eqref{ch1sys:fieldroad} to the one of the system
\begin{equation}\label{ch1sys:symmetric}
\left\{
\begin{array}{ll}
v_t-d\Delta v - c\partial_x v= f(x,v),  & (x,y)\in\Omega,\\
-\partial_y v|_{y=0} =0, & x\in\R,
\end{array}
\right.
\end{equation}
whose solution is a function $v(x,y)$ that can be extended by symmetry to the whole plane, thanks to the Neumann border condition.
It is natural to consider system \eqref{ch1sys:symmetric} as the counterpart of system \eqref{ch1sys:fieldroad} in the case without the road, since it presents the same geometry, including the same boundary condition exception made for the exchange terms that are in place for the case of a fast diffusion channel.

\subsection{State of the art}

We present here the background that led us consider system \eqref{ch1sys:fieldroad} and some useful results that are known in the community.

The study of reaction-diffusion equations started with the works by Fisher \cite{fisher} and by Kolmogorov, Petrowskii and Piskunov \cite{KPP}, who modelled the spacial diffusion of an advantageous gene in a population living in a one-dimensional environment through the equation
\begin{equation}\label{ch1eq:KPP}
\partial_t v -d \, \partial_{xx}^2 v = f(v) 
\end{equation}
for $x\in\R$ and $t\geq 0$. For \eqref{ch1eq:KPP}, it is supposed that $d>0$ and $f\geq 0$ is a $\mathcal{C}^1$ function satisfying $f(0)=f(1)=0$ and the KPP hypothesis $f(v)\leq f'(0)v$ for $v\in[0,1]$. The first example was a nonlinearity of logistic type, so $f(v)= av(1-v)$ for some $a>0$.
It was shown any solution $v$ issued from a nonnegative initial datum $v_0$ converges to 1 as $t$ goes to infinity, locally uniformly in space; this long time behaviour is called \emph{invasion}. 
The generalisation in higher dimension of equation \eqref{ch1eq:KPP} was then used to study the spatial diffusion of animals, plants, bacteria and epidemics \cite{skellam, okubo1980diffusion}.

A vast literature has been originated from the pioneer works, studying various aspects of the homogeneous equation \eqref{ch1eq:KPP}, in particular concerning the \emph{travelling fronts}. These are solutions of the form $v(t,x)= V(x \cdot e +ct)$ with $V:\R\to[0,1]$, for $e$ a direction, the \emph{direction of propagation},  and $c$ the \emph{speed of propagation of the travelling front}. Other than this, researchers have investigated the \emph{asymptotic speed of propagation} at which level sets of a solution starting from $v_0$ expands. These topics arose already in \cite{fisher} and \cite{KPP}, and their investigation was continued in many interesting articles, among which \cite{fife} and \cite{weinberger2}.

The correspondence of the theoretical results with actual data as seen in \cite{skellam} was encouraging, however it was clear that natural environments, even at macroscopic levels, were not well represented by a homogeneous medium, due to the alternation of forests, cultivated fields, plains, scrubs and many other habitats, as well as roads, rivers and other barriers \cite{kinezaki2003modeling}. It was necessary to look at more sophisticated features, as the effects of inhomogeneity, fragmentation, barriers and fast diffusion channels, and on the top of that, climate change.

A first analysis was carried out %by Shigesada, Kawasaki and Teramoto 
in \cite{shigesada1986traveling, shigesada1997biological} and %then also by Kinezaki
in \cite{kinezaki2003modeling} for the so-called the \emph{patch model}. The authors considered a periodic mosaic of two different homogeneous habitats, one favorable and one unfavorable for the invading species. 
In \cite{kinezaki2003modeling}, the authors studied the long time behaviour of the population starting from any nonnegative initial datum. For further convenience, let us give the following definition:
\begin{definition}\label{ch1def:pers_extin}
	For the equation of \eqref{ch1eq:KPP} or the system \eqref{ch1sys:fieldroad}, we say that 
	\begin{enumerate}
		\item \emph{extinction} occurs if any solution starting from a non negative bounded initial datum converges to $0$ or to $(0,0)$  uniformly as $t$ goes to infinity.
		\item \emph{persistence} occurs if any solution starting from a non negative, non zero, bounded initial datum converges to a positive  stationary solution locally uniformly as $t$ goes to infinity.
	\end{enumerate}
\end{definition}

In \cite{kinezaki2003modeling}, it was empirically shown that the stability of the trivial solution $v=0$ determines the long time behaviour of the solutions. A solid mathematical framework for a general periodic environment was given in \cite{bhroques}. There, the authors considered the equation
\begin{equation}\label{ch1eq:bhroques}
\partial_t v - \nabla \cdot (A(x)\cdot \nabla v) = f(x, v)
\end{equation}
for $x\in \R^N$ and $t\geq 0$. 
The diffusion matrix $A(x)$ is supposed to be $\mathcal{C}^{1, \alpha}$, uniformly elliptic and periodic; however, for our interest we can suppose $A(x)=d\, I_{N}$, where $I_N$ is the identity matrix. The nonlinearity $f: \R^N \times \R_{\geq0} \to \R$ is supposed to be $\mathcal{C}^{1}$ in $x$, locally in $v$, and Lipshitz in $v$, uniformly in $x$, respecting hypothesis \eqref{ch1hyp:0}-\eqref{ch1hyp:KPP} and such that for some $L=(L_1, \dots, L_N)$, with $L_i\geq 0$, it holds
\begin{equation}\label{ch1hyp:per'}
f(x+L,s)=f(x,s) \quad \text{for all} \  s\geq 0 \ \text{and all} \ x\in\R^N.
\end{equation}

The criterion for persistence or extinction is given via a notion of periodic eigenvalue, that is the unique number $\lambda_p(-\mathcal{L}, \R^N)$
such that there exists a solution $\psi\in W_{loc}^{2, p}(\R^N)$ to the system
\begin{equation}\label{ch1sys:L_RN_p}
\left\{
\begin{array}{ll}
\mathcal{L'}(\psi) + \lambda \psi = 0, & x\in\R^N, \\
\psi> 0, &  x\in\R^N, \\
|| \psi ||_{\infty}=1, \\
\psi \ \text{is periodic in $x$ of periods $L$},
\end{array}
\right.
\end{equation}
where $\mathcal{L'}$ is given by
\begin{equation}\label{ch1def:mathcal_L'}
\mathcal{L'}(\psi):=  d \Delta \psi  + f_v(x,0)\psi.
\end{equation}
We point out that the existence and uniqueness of $\lambda_p(-\mathcal{L}, \R^N)$ is guaranteed by Krein-Rutman theory. The long time behaviour result in \cite{bhroques} is the following:

\begin{theorem}[Theorem 2.6 in \cite{bhroques}]\label{ch1thm:2.6inbhroques}
	Assume $f$ satisfies \eqref{ch1hyp:0}-\eqref{ch1hyp:KPP} and \eqref{ch1hyp:per'}. Then: 
	\begin{enumerate}
		\item If $\lambda_p(-\mathcal{L'}, \R^N)<0$, persistence occurs for \eqref{ch1eq:bhroques}.
		\item If $\lambda_p(-\mathcal{L'}, \R^N)\geq 0$, extinction occurs for \eqref{ch1eq:bhroques}.
	\end{enumerate}
\end{theorem}

To prove Theorem \ref{ch1thm:2.6inbhroques}, the authors performed an analysis of $\lambda_p(-\mathcal{L}, \R^N)$, proving that it coincide with the limit of eigenvalues for a sequence of domains invading $\R^N$, so that it coincides with the generalised principal eigenvalue of the system ``without the road'' \eqref{ch1sys:symmetric}. Nowadays, that and many other properties of this eigenvalue can be found as part of a broader framework in \cite{br}. In Section \ref{ch1s:eigenvalues}, we will provide further comments on it.

Another important fact highlighted both in the series in \cite{shigesada1986traveling, shigesada1997biological, kinezaki2003modeling} and in \cite{bhroques} is that the presence of multiple small unfavourable zones gives less chances of survival than one large one, the surface being equal. 

\medskip

A new difficulty that one may consider while studying ecological problems is, sadly, the issue of a changing climate. A 1-dimensional model in this sense was first proposed in
\cite{berestycki2009can} and \cite{potapov2004climate}, and was later treated in higher dimension in \cite{br2}. 
The authors first imagined that a population lives in a favourable region enclosed into a disadvantageous environment; due to the climate change, the favourable zone starts to move in one direction, but keeps the same surface. The resulting equation is 
\begin{equation}\label{ch11709}
\partial_t v -  \Delta v=f(x-ct e,v) \quad \text{for} \ x\in\R^N,
\end{equation}
with $e$ a direction in $\mathbb{S}^{N-1}$ and $f: \R^N\times \R_{\geq 0} \to \R$. It was observed that a solution to \eqref{ch11709} in the form of a travelling wave $v(x,t)=V(x-cte)$ solves the equation
\begin{equation}\label{ch1eq:cc}
\partial_t V -  \Delta V- c\,e\cdot \nabla V=f(x,V) \quad \text{for} \ x\in\R^N,
\end{equation}
which is more treatable. The main question is if the population keeps pace with the shifting climate, that is, if the species is able to migrate with the same speed of the climate. The answer to this question is positive if a solution to \eqref{ch1eq:cc} exists; this depends on the value of $c$. We point out that already in \cite{br2} the authors considered the general case of a possible periodic $f(x,v)$.

\medskip

As mentioned before, another feature worth investigation is the effect of fast diffusion channels on the survival and the spreading of species. In fact, the propagation of invasive species as well as epidemics is influenced by the presence of roads \cite{robinet2012human, gatto2020spread}.
This observations led Berestycki, Roquejoffre and Rossi to propose a model for ecological diffusion in the presence of a fast diffusion channel in \cite{brr}, the so-called \emph{road-field model}. The field is modelled with the halfplane $\Omega=\R \times \R_{+}$ and the line with the $x$ axis; the main idea is to use two different variables for modelling the density of population along the line, $u$, and on the half plane, $v$. The system reads
\begin{equation*} 
\left\{
\begin{array}{ll}
\partial_t u(x,t) - D \partial_{xx}^2 u (x,t) = \nu v (x,0,t) - \mu u(x,t), &  x\in \R, t > 0, \\
\partial_t v(x,y,t) - d \Delta v (x,y,t)= f(v), & (x,y) \in \Omega, t>0, \\
-d \partial_y v(x,0,t) = -\nu v(x,0,t) + \mu u(x,t), & x \in \R, t>0, 
\end{array} \right.
\end{equation*}
for $D$, $d$, $\nu$, $\mu$ positive constants; moreover, $f\in \mathcal{C}^1$ was supposed to satisfy
\begin{equation*}
f(0)=f(1)=0, \quad 0< f(s) < f'(0)s \ \text{for} \ s \in (0,1), \quad f(s)<0 \ \text{for}  \ s>1.
\end{equation*}
The three equations describe, respectively, the dynamic on the line, the dynamic on the half plane and the exchanges of population between the line and the half plane. On the line, the diffusion is faster than in  $\Omega $ if $D>d$. 
In \cite{brr}, the authors identify the unique positive stationary solution $\left(\frac{1}{\mu}, 1 \right)$ and prove persistence of the population. 
Moreover, they show that the presence of the line increases the spreading speed. Another version of the model with a reaction term for the line was presented by the same authors in \cite{BRR2}, while many variation of the models were proposed by other authors: with nonlocal exchanges in the direction of the road \cite{pauthier2015uniform, pauthier2016influence}, with nonlocal diffusion \cite{berestycki2015effect, berestycki2014speed}, and with different geometric settings \cite{rossi2017effect}. For a complete list, we refer to \cite{tellini2019comparison}.

The case of heterogeneous media for systems of road-field type has been so far not much treated, due to its difficulties. A first road-field model with exchange terms that are periodic in the direction of the road was proposed in \cite{giletti2015kpp}. There, the authors 
recovered the results of persistence and of acceleration on the propagation speed due to the road known in the homogeneous case;
they also studied the spreading of solution with exponentially decaying initial data and calculated their speeds.

Recently, Berestycki, Ducasse and Rossi introduced in \cite{romain} a new generalised principal eigenvalue fitting road-field system for possibly heterogeneous reaction term; here, we give its definition directly for the system \eqref{ch1sys:fieldroad}.  
Calling
\begin{equation}\label{ch1sys:operators}
\left\{
\begin{array}{l}
\mathcal{R}(\phi, \psi):=D \phi''+c \phi'+\nu {\psi}|_{y=0}-\mu \phi, \\
\mathcal{L}(\psi):= d\Delta \psi +c \partial_x \psi -f_v(x,0)\psi, \\
B(\phi, \psi):=d \partial_y {\psi}|_{y=0}+\mu \phi- \nu {\psi}|_{y=0},
\end{array}
\right.
\end{equation}
this eigenvalue is defined as 
\begin{equation}\label{ch1def:lambda1_S_Omega}
\begin{split}
\lambda_1( \Omega)=\sup \{ \lambda \in \R \ : \ \exists (\phi, \psi)\geq (0,0), \ (\phi, \psi) \not\equiv(0,0), \ \text{such that} \\ \mathcal{L}(\psi) + \lambda \psi \leq 0 \ \text{in} \ \Omega, \ \mathcal{R}(\phi, \psi) +\lambda \phi \leq 0  
\ \text{and} \ B(\phi, \psi)\leq 0 \ \text{in} \ \R \},
\end{split}
\end{equation}
with $(\phi, \psi)$ belonging to $W_{loc}^{2,3}(\R)\times W_{loc}^{2,3}(\overline{\Omega})$. Together with the definition, many interesting properties and bounds were studied; we will recall some of them later.

Thanks to that, the same authors were able to investigate the case of 
a favourable ecological niche, possibly facing climate change, in \cite{econiches}. It was proven that the sign of $\lambda_1( \Omega)$ characterises
the extinction or the persistence of the population; moreover, comparing the results with  the ones found for the model without the road, a deleterious effect of the road on the survival chances is always found when there is no climate change. On the other hand, if the ecological niche shifts, the road has in some cases a positive effect on the persistence.

\subsection{Main results}

We are now ready to present the main results of this paper.

\subsubsection{The case of a periodic $f(x,v)$}

Here, we consider the case of a nonlinearity that respects the KPP hypothesis and is periodic in the direction of the road. Moreover, here we always consider $c=0$.

We begin by the following result on long time behaviour for solutions of system \eqref{ch1sys:fieldroad}:

\begin{theorem} \label{ch1thm:char}
	Assume $f$ satisfy \eqref{ch1hyp:0}-\eqref{ch1hyp:per}, $c=0$ and let $\lambda_1(\Omega)$ be as in \eqref{ch1def:lambda1_S_Omega}.
	Then the following holds:
	\begin{enumerate}
		\item if $\lambda_1( \Omega)\geq 0$, then  extinction occurs.
		\item if $\lambda_1(\Omega)<0$, then persistence occurs and the  positive stationary solution $(u_{\infty}, v_{\infty})$ is unique and periodic in $x$.
	\end{enumerate} 
\end{theorem}

Now, we compare the behaviour of solutions to the system \eqref{ch1sys:fieldroad} with the ones of system \eqref{ch1sys:symmetric}. 
This allows us to highlight the effects of the fast diffusion channel on the survival chances of the population.
Actually, since solutions of \eqref{ch1sys:symmetric} can be extended by refection to the whole plane, we can make the comparison with equation \eqref{ch1eq:bhroques} for $A(x)=d I_2$ and $L=(\ell, 0)$.
The comparison is performed thanks to the generalised principal eigenvalue $\lambda_1(\Omega)$ for system \eqref{ch1sys:fieldroad} and the periodic eigenvalue $\lambda_p(-\mathcal{L}, \R^2)$, as defined in \eqref{ch1sys:L_RN_p}, for the operator $\mathcal{L}$ in dimension 2.
We obtain the following:

\begin{theorem}\label{ch1thm:comparison}
	Assume $f$ respects hypothesis \eqref{ch1hyp:0}-\eqref{ch1hyp:per}, $c=0$. Then:
	\begin{enumerate}
		\item if  $\lambda_p(-\mathcal{L}, \R^2)<0$, then $\lambda_1( \Omega)<0$, that is, if persistence occurs  for the system ``without the road'' \eqref{ch1eq:bhroques}, then it occurs also for system ``with the road'' \eqref{ch1sys:fieldroad}.
		\item if $\lambda_p(-\mathcal{L}, \R^2)\geq 0$, then $\lambda_1( \Omega)\geq 0$, that is, if extinction occurs for the system ``without the road''  \eqref{ch1eq:bhroques}, then it occurs also for system ``with the road'' \eqref{ch1sys:fieldroad}.
	\end{enumerate}	
\end{theorem}

%\begin{proposition}\label{ch1prop:c=0}
%	If $c=0$ and $f\equiv g$, then $\lambda_1( \Omega)$ and $\lambda_p(-\mathcal{L}, \R^n)$ have the same sign, namely
%	 $\lambda_1( \Omega)>0$ implies $\lambda_p(-\mathcal{L}, \R^n)>0$ and $\lambda_1(\Omega)<0$ implies $\lambda_p(-\mathcal{L}, \R^n)<0$.
%\end{proposition}
%vedi prop 3.1 di romain in ecological niche

Theorem \ref{ch1thm:comparison} says that the road has no negative impact on the survival chances of the population in the case of a periodic medium depending only on the variable in the direction of the road.  
This is surprising if compared to the results obtained in \cite{econiches} (precisely Theorem 1.5, part \emph{(ii)}), where the authors find that the existence of the road is deleterious in presence of an ecological niche, and even more counter-intuitive owing the fact that fragmentation of the environment lessens the survival chances of a population, as shown in \cite{bhroques}. This means that, in the case of periodic media, the presence of the fast diffusion channel does not interfere with the persistence of the population, which depends only on the  environment of a periodicity cell.
As seen in \cite{bhroques}, where the dependence of persistence on the amplitude of fragmentation was studied, if the favourable zones are sufficiently large, the population will eventually spread in all of them; the presence of the road does not cause loss of favourable environment and consequently of persistence chances.
However, we expect the spreading speed to be influenced by the presence of the road, as it has been already proven in the case of homogeneous environment.

We point out that Theorem \eqref{ch1thm:char} completes and is in accordance with the results on long time behaviour found in \cite{brr} for a homogeneous reaction term, which we can see as a particular case of periodicity, which respects positive KPP hypothesis (where the positivity is requested through $f'(0)>0$). In \cite{brr}, Theorem 4.1 states the convergence of any positive solution to the unique positive stationary solution of the system.  Since it is well known that for the homogeneous case it holds $\lambda_1(-\mathcal{L}, \R^2)=- f'(0)$, the positivity hypothesis gives that $\lambda_1(-\mathcal{L}, \R^2)<0$ and, as a consequence of Theorem \ref{ch1thm:ineq}, that the second case in our Theorem \ref{ch1thm:char} occurs. 
If instead we asked for $f'(0)\leq0$, then we would be in the first case of Theorem \ref{ch1thm:char}, yielding extinction of the population.

\paragraph{Effects of amplitude of heterogeneity.}
One may expect that the presence of a road may alter the complex interaction between more favourable and less favourable zones, in particular penalising the persistence, since it was shown that populations prefer a less fragmented environment. However, the road does not interfere with that; as a consequence, also for environments presenting fast diffusion channels, some results of the analysis on the effect of fragmentation performed in \cite{bhroques} holds.

Take a parameter $\alpha>0$ and consider system \eqref{ch1sys:fieldroad} with nonlinearity
\begin{equation}\label{ch11421}
\tilde{f}(x,v)=\alpha f(x,v).
\end{equation}
To highlight the dependence on $\alpha$, we will call $\lambda_1(\Omega, \alpha)$ the  generalised principal eigenvalue defined in \eqref{ch1def:lambda1_S_Omega} with nonlinearity $\tilde{f}$.
As a direct consequence of Theorem \eqref{ch1thm:comparison} and Theorem 2.12 in \cite{bhroques}, we have the following result on the amplitude of heterogeneity:

\begin{corollary}
	Assume $\tilde{f}$ is defined as in \eqref{ch11421}, $f$ satisfies \eqref{ch1hyp:0}-\eqref{ch1hyp:per}, and $c=0$. Then:
	\begin{enumerate}
		\item if $ \int_{0}^{\ell} f_v(x,0)>0$, or if $ \int_{0}^{\ell} f_v(x,0)=0$ and $f\not\equiv 0$, then for all $\alpha >0$ we have $\lambda_1(\Omega, \alpha  )<0$.
		\item if $ \int_{0}^{\ell} f_v(x,0)<0$, then  $\lambda_1(\Omega, \alpha )>0$ for $\alpha$ small enough; if moreover there exists $x_0\in[0,\ell]$ such that $f_v(x_0,0)>0$, then for all $\alpha$ large enough $\lambda_1(\Omega, \alpha )<0$.
	\end{enumerate}
\end{corollary}

This result describes with precision the fact that, to persist, a species must have a sufficiently large favourable zone available. If the territory is more advantageous than not, then the population persist. If however there environment is generally unfavourable, the population persists only if there are some contiguous advantageous zones large enough; if instead the advantageous zones are fragmented, even if there is unlimited favourable territory, the population will encounter extinction.

\subsubsection{A climate change setting for a general $f(x,v)$}

We consider now a general nonlinearity that depends on the spatial variable in the direction of the road. We stress the fact that we do not suppose any periodicity, but the case of a periodic $f$ is a particular case of this setting. Moreover, the following result is done in the general framework of a possible climate change, so the parameter $c$ may be different from $0$.

Comparison between the systems with and without the road, in the general case, are done through comparison between $\lambda_1(\Omega)$ and the generalised principal eigenvalue of system \eqref{ch1sys:symmetric}, given by
\begin{equation}\label{ch1lambda:L_Omega}
\begin{split}
\lambda_1(-\mathcal{L}, \Omega)=\sup \{ \lambda \in \R \ : \ \exists \psi \geq 0, \psi \not\equiv 0 \ \text{such that} \\ 
\mathcal{L}(\psi) + \lambda \psi \leq 0 \ \text{on} \ \Omega, \ -\partial_y \psi|_{y=0}\leq 0 \ \text{on} \ \R \}
\end{split}
\end{equation}
for $\psi\in W_{loc}^{2,3}(\Omega)$. With this notation, we have the following:

\begin{theorem}\label{ch1thm:ineq}
	Assume $\lambda_1(-\mathcal{L}, \R^2)$ as in \eqref{ch1lambda:L_Omega} and $\lambda_1(\Omega)$ as in \eqref{ch1def:lambda1_S_Omega}; then $\lambda_1(-\mathcal{L}, \R^2) \geq \lambda_1(\Omega)$.
\end{theorem}

In the special case $c=0$, some information on the relations between $\lambda_1(-\mathcal{L}, \R^2)$ and $\lambda_1(\Omega)$  was already available in \cite{econiches}: Proposition 3.1 yields that $\lambda_1(-\mathcal{L}, \R^2)\geq 0$ implies $\lambda_1(\Omega)\geq 0$. Thanks to that and Theorem \ref{ch1thm:ineq}, the following result holds:

\begin{corollary}\label{ch1thm:ineq2}
	If $c=0$, we have $\lambda_1(-\mathcal{L}, \R^2)<0$ if and only if $\lambda_1(\Omega)<0$.
\end{corollary}

As already pointed out in \cite{romain}, even for $c=0$ it is not true that $\lambda_1(-\mathcal{L}, \R^2) =\lambda_1(\Omega)$. In fact, it has been found that $\lambda_1(\Omega) \leq \mu$, while playing with $f$ one can have $\lambda_1(-\mathcal{L}, \R^2)$ as large as desired. However, the fact that the two eigenvalues have the same sign reveals that they are profoundly linked.

%In the general framework, i.e. when we drop the periodicity assumption on $f$, to be able to characterise survival or extinction of the population via the generalised principal eigenvalue, one would need more than just its sign. In fact, what is usually proved is the existence of a positive stationary solution and the convergence of any solution to it when the eigenvalue has a negative sign; we mention that as an open problem for a general heterogeneous reaction function independent of $y$. 

\subsection{Organisation of the paper}

In Section \ref{ch1s:eigenvalues}, we recall and discuss the properties of the eigenvalues $\lambda_1(\Omega)$, $\lambda_1(-\mathcal{L}, \R^2)$ and $\lambda_p(-\mathcal{L}, \R^2)$ already known in the literature. 
Furthermore, a periodic eigenvalue for the system \eqref{ch1sys:fieldroad} will be defined; because of the presence of the road, the periodicity is present only in the $x$ direction. As a consequence, it is useful to define an analogous generalised eigenvalue for the system without the road \eqref{ch1sys:symmetric} with periodicity only in the direction of the road.

In Section \ref{ch1s:ordering}, one finds the proof of Theorem \ref{ch1thm:ineq} and Theorem \ref{ch1thm:comparison}. Moreover, the relations between the newly defined generalised periodic eigenvalues and the known ones are shown.

The last Section \ref{ch1s:lb} treats large time behaviour for solutions to \eqref{ch1sys:fieldroad} with $c=0$ and periodic $f$; this includes the proof of Theorem \ref{ch1thm:char}.

\section{Generalised principal eigenvalues and their properties} \label{ch1s:eigenvalues}

Both road-field model and reaction-diffusion equations in periodic media have been treated in several papers.
In this section, we introduce some useful objects and recall their properties. 
All along this section we will make repeated use of the operators $\mathcal{L}$, $\mathcal{R}$ and $B$, that were  defined in \eqref{ch1sys:operators}.

\subsection{Eigenvalues in periodic media}

Since $\mathcal{L}$ has periodic terms, it is natural to look for eigenfunctions that have the same periodicity. However, to begin the discussion on the periodic eigenvalue for the operator $\mathcal{L}$ in $\R^2$, we consider  its counterpart in $\R$. 
We look for the unique number $\lambda_p(-\mathcal{L}, \R)\in\R$ such that there exists a function $\psi\in W_{loc}^{2, p}(\R)$ solution to the problem
\begin{equation}\label{ch1sys:L_R_p}
\left\{
\begin{array}{ll}
d\psi''+f_v(x, 0)\psi + \lambda \psi = 0, & x\in\R, \\
\psi> 0, &  x\in\R, \\
|| \psi ||_{\infty}=1, \\
\psi \ \text{is periodic in $x$ of period $\ell$.}
\end{array}
\right.
\end{equation}
In \eqref{ch1sys:L_R_p}, the operator $\mathcal{L}$ has been replaced by an operator working on $\R$, namely the Laplacian has been substituted by a double derivative.
Notice that existence and uniqueness of the solution to \eqref{ch1sys:L_R_p}, that we call $(\lambda_p(-\mathcal{L}, \R), \psi_p)$, is guaranteed  by Krein-Rutman theory.

For the operator $\mathcal{L}$, since it has no dependence on the $y$ variable,
we have to introduce a fictive periodicity in order to be able to use the Krein-Rutman theory. Thus, fix $\ell'>0$ and consider the problem in $\R^2$ of finding the value $\lambda_p(-\mathcal{L}, \R^2)\in\R$ such that there exists a solution $\psi\in W_{loc}^{2, p}(\R^2)$ to the system
\begin{equation}\label{ch1sys:L_R2_p}
\left\{
\begin{array}{ll}
\mathcal{L}(\psi) + \lambda \psi = 0, & (x,y)\in\R^2, \\
\psi> 0, &  (x,y)\in\R^2, \\
|| \psi ||_{\infty}=1, \\
\psi \ \text{is periodic in $x$ and $y$ of periods $\ell$ and $\ell'$}.
\end{array}
\right.
\end{equation}
Again we can use the Krein-Rutman theorem to see that there exists a unique couple $(\lambda(-\mathcal{L}, \R^2), \psi_{\ell'})$ solving \eqref{ch1sys:L_R2_p}. Now, with a slight abuse of notation, we consider the function $\psi_p(x,y)$ as the extension in $\R^2$ of $\psi_p$ solution to \eqref{ch1sys:L_R_p}. We observe that the couple $(\lambda_p(-\mathcal{L}, \R), \psi_p)$ gives a solution to \eqref{ch1sys:L_R2_p}. Hence, 
\begin{equation}\label{ch1eq:-5}
\lambda_p(-\mathcal{L}, \R^2)=\lambda_p(-\mathcal{L}, \R) \quad \text{and} \quad \psi_p\equiv \psi_{\ell'}.
\end{equation}
This also implies that neither $\lambda_p(-\mathcal{L}, \R^2)$ nor $\psi_{\ell'}$ depend on the parameter $\ell'$ that was artificially introduced. From now on, we will use only $ \psi_p$.

The properties of the eigenvalue $\lambda_p(-\mathcal{L}, \R^2)$ were also studied in \cite{br}, where it is called $\lambda'$ and defined as
\begin{equation}\label{ch1def:lambdap_dim2}
\begin{split} 
\lambda_p(-\mathcal{L}, \R^2)= &\inf \{ \lambda \in \R \ :\  \exists \varphi\in \mathcal{C}^2(\R^2)\cap L^{\infty}(\R^2), \ \varphi>0,  \\ &\hspace{8em} \varphi \ \text{periodic in $x$ and $y$}, \ \mathcal{L}(\varphi)+\lambda \varphi \geq 0    \}.
\end{split}
\end{equation}
In particular, in Proposition 2.3 of \cite{br} it is stated that the value found with \eqref{ch1sys:L_R2_p} coincides with the one defined in \eqref{ch1def:lambdap_dim2}.

\subsection{Generalised principal eigenvalues for the system with and without the road and some properties}

In this section, we are going to treat eigenvalues that are well defined also for non periodic reaction functions. 

The generalised eigenvalue $\lambda_1( \Omega)$ for the system \eqref{ch1sys:fieldroad}, that we defined in \eqref{ch1def:lambda1_S_Omega}, was first introduced in \cite{romain}. Together with this, the authors also proved the interesting property that $\lambda_1( \Omega)$ coincides with the limit of principal
eigenvalues of the same system restricted to a sequence of invading domains. 
They use some half ball domains defined as follow for $R>0$:
\begin{equation}\label{ch11722}
\Omega_R:=B_R\cap (\Omega) \quad \text{and} \quad I_R:=(-R, R).
\end{equation}
Them we have the following characterisation for $\lambda_1( \Omega)$:
\begin{proposition}[Theorem 1.1 of \cite{romain}] \label{ch1prop:romain}
	For $R>0$, 
	there is a unique 
	$\lambda_1( \Omega_R) \in \R$
	and a unique (up to multiplication by a positive scalar) positive 
	$(u_R, v_R) \in W^{2,3}(I_R) \times W^{2,3} (\Omega_R)$ that satisfy the eigenproblem
	\begin{equation}\label{ch1sys:halfball}
	\left\{
	\begin{array}{ll}
	\mathcal{R}(\phi, \psi) +\lambda\phi = 0, & x\in I_R, \\
	\mathcal{L}(\psi)  + \lambda \psi = 0, &(x,y)\in \Omega_R, \\
	B(\phi, \psi)= 0, & x\in I_R, \\
	\psi =0, & (x,y)\in (\partial\Omega_R ) \setminus (I_R\times \{0\}) \\
	\phi(R)=\phi(-R)=0. &
	\end{array}
	\right.
	\end{equation}
	Moreover, 
	\begin{equation*}
	\lambda_1( \Omega_R) \underset{R\to +\infty}{\searrow} \lambda_1( \Omega).
	\end{equation*}
\end{proposition}

We point out that, using the strong maximum principle, we have that $u_R > 0$ in $I_R$ and $v_R > 0$ in $\Omega_R$.

%In \cite{romain}, the authors study the existence of a bounded $\lambda_1( \Omega)$ and of a couple of non-negative eigenfunctions $(\phi_1(x), \psi_1(x,y))\not \equiv (0,0)$ that satisfies
%\begin{equation*}
%\left\{
%\begin{array}{ll}
%\mathcal{R}(\phi_1, \psi_1)  +\lambda_1 \phi_1 = 0, & x\in \R, \\
%\mathcal{L}(\psi_1) + \lambda_1 \psi_1 = 0, & (x,y)\in \Omega, \\
%B(\phi_1, \psi_1)= 0, & x\in \R.
%\end{array}
%\right.
%\end{equation*}
%The required regularity is $\phi\in W_{loc}^{2,3}(\R)$ and $\psi\in  W_{loc}^{2,3}(\R \times [0, +\infty))$.
%Existence of such triple is a by-product of Proposition 2.6 in \cite{romain}.

We also consider the principal eigenvalue on the truncated domains for the linear operator $\mathcal{L}(\psi)$. 
To do that, for any $R>0$ we call $B_R^P$ the ball of centre $P=(x_P,y_P)$ and radius $R$.  We define $\lambda_1(-\mathcal{L}, B_R^P)$ as the unique real number such that
the problem
\begin{equation}\label{ch1sys:L_BR}
\left\{
\begin{array}{ll}
\mathcal{L}(\psi_R) + \lambda_1(-\mathcal{L}, B_R^P) \psi_R = 0, & (x,y)\in B_R^P, \\
%\partial_y \psi_R= 0, & x\in I_R, \\
\psi_R=0, & (x,y)\in \partial B_R^P, %\setminus (I_R\times \{0\})
\end{array}
\right.
\end{equation}
admits a positive solution $\psi_R\in W^{2,3}(B_R^P)$.
The existence and uniqueness of such quantity and its eigenfunction is a well-known result derived via the Krein-Rutman theory. 
We also notice that, calling $B_R$ the ball with radius $R$ and center $O=(0,0)$, the couple $(\lambda_1(-\mathcal{L}, B_R), \psi_R)$ is also a solution to the problem
\begin{equation}\label{ch1sys:L_OmegaR}
\left\{
\begin{array}{ll}
\mathcal{L}(\psi) + \lambda \psi = 0, & (x,y)\in \Omega_R, \\
\partial_y \psi= 0, & x\in I_R, \\
\psi=0, & (x,y)\in (\partial \Omega_R)\setminus (I_R\times \{0\}).
\end{array}
\right.
\end{equation}
The proof of that is very simple. 
If $(\lambda, \psi)$ is the unique solution to \eqref{ch1sys:L_OmegaR},  by extending $\psi$ by symmetry in $B_R$ we get a solution to \eqref{ch1sys:L_BR}. By the uniqueness of the solution to \eqref{ch1sys:L_BR}, we get $\lambda=\lambda_1(-\mathcal{L}, B_R)$.

Similarly to what happens with $\lambda_1( \Omega_R)$, we have that $\lambda_1(-\mathcal{L}, B_R)$ converges to the value $\lambda_1(-\mathcal{L}, \Omega)$ defined in \eqref{ch1lambda:L_Omega}, as stated in Proposition 2.4 of \cite{econiches}, that we report here:

\begin{proposition}[Proposition 2.4 in \cite{econiches}] \label{ch1prop:lim_B_R}
	It holds \begin{equation*}
	\lambda_1(-\mathcal{L}, B_R) \underset{R\to +\infty}{\searrow} \lambda_1(-\mathcal{L}, \Omega).
	\end{equation*}
\end{proposition}

Another notion of generalised eigenvalue analysed in \cite{br} is the quantity 
\begin{equation}\label{ch1lambda:L_R2}
\begin{split}
\lambda_1(-\mathcal{L}, \R^2)=\sup \{ \lambda \in \R \ : \ \exists \psi \geq 0, \psi \not\equiv 0 \ \text{such that}  \
\mathcal{L}(\psi) + \lambda \psi \leq 0 \ \text{a.e on} \ \R^2 \}
\end{split}
\end{equation}
for test functions $\psi \in W_{loc}^{2,p}(\R^2)$. As stated in Proposition 2.2 of \cite{br},
$\lambda_1(-\mathcal{L}, \R^2)$ coincides with the limit of the sequence eigenvalues $\lambda_1(-\mathcal{L}, B_n)$.
By that and Proposition \ref{ch1prop:lim_B_R}, we have 
\begin{equation*}
\lambda_1(-\mathcal{L}, \R^2)=\lambda_1(-\mathcal{L}, \Omega)
\end{equation*}
With this notation, we can report the following affirmations deriving from Theorem 1.7 in \cite{br} for the case of a periodic reaction function:

\begin{theorem}[Theorem 1.7 in \cite{br}]\label{ch1thm:1.7inbr}
	Suppose $f$ satisfies \eqref{ch1hyp:per}.
	The following holds:
	\begin{enumerate}
		\item It holds that $\lambda_p(-\mathcal{L}, \R^2)\leq \lambda_1(-\mathcal{L}, \Omega)$.
		\item If $\mathcal{L}$ is self-adjoint (i.e, if $c=0$), then $\lambda_p(-\mathcal{L}, \R^2)=\lambda_1(-\mathcal{L}, \Omega)$.
	\end{enumerate}
\end{theorem}

At last, we recall the following result on the signs of the eigenvalues for the systems with and without the road:
\begin{proposition}[Proposition 3.1 in \cite{econiches}]\label{ch1prop:ineq}
	It holds that
	\begin{equation*}
	\lambda_1(-\mathcal{L}, \Omega) \geq 0 \quad \Rightarrow \quad \lambda_1( \Omega) \geq 0.
	\end{equation*} 
\end{proposition}
This is the result that, in combination with Theorem \ref{ch1thm:ineq}, gives Corollary \ref{ch1thm:ineq2}.

\subsection{The periodic generalised principal eigenvalue for the road-field system}

We introduce here two new eigenvalues that will be useful in the following proofs. They are somehow of mixed type, in the sense that the domains in which they are defined are periodic in the variable $x$ and truncated in the variable $y$. Here, we require $f$ to be periodic as in hypothesis \eqref{ch1hyp:per}. 

Given $r>0$, let  $(\lambda_p(-\mathcal{L}, \R\times(-r, r)), \psi_{r})$ be the unique couple solving the eigenvalue problem
\begin{equation}\label{ch1sys:bary2}
\left\{
\begin{array}{ll}
\mathcal{L}(\psi_{r})  + \lambda \psi_{r} = 0, \qquad(x,y)\in \R \times (-r, r), \\
\psi_{r} (x, \pm r)=0, \qquad x\in \R, \\
||\psi_{r}||_{\infty}=1, \ \psi_{r} \ \text{is periodic in} \ x.
\end{array}
\right.
\end{equation}
The existence and uniqueness  of the solution to \eqref{ch1sys:bary2} derives once again from Krein-Rutman theory.

We point out that $\lambda_p(-\mathcal{L}, \R\times(-r,r))$ is decreasing in $r$ by inclusion of domains. So,
there exists a well defined number, that with a slight abuse of notation we call $\lambda_p (-\mathcal{L}, \Omega)$, such that
\begin{equation}\label{ch1eq:-2}
\lambda_p(-\mathcal{L}, \R\times(-r,r)) \underset{r\to+\infty}{\searrow} \lambda_p (-\mathcal{L}, \Omega).
\end{equation}

% autovalore con la strada periodico
Given $r>0$, there exists a unique  value $\lambda_p( \R\times(0, r))\in\R$ such that the problem
\begin{equation}\label{ch1sys:r}
\left\{
\begin{array}{ll}
\mathcal{R}(\phi, \psi) +\lambda\phi = 0, \qquad x\in \R, \\
\mathcal{L}(\psi)  + \lambda \psi = 0, \qquad(x,y)\in \R \times (0, r), \\
B(\phi, \psi)= 0, \qquad x\in \R, \\
\psi (\cdot, r)=0, \\
\phi \ \text{and} \ \psi  \ \text{are periodic in} \ x,
\end{array}
\right.
\end{equation}
has a solution.
The proof of the existence can be derived by modifying for periodic functions the proof of the existence of $\lambda_1( \Omega_R)$ that is found in the Appendix of \cite{romain}.

Moreover, we define
\begin{equation*}
\begin{split}
\lambda_p ( \Omega)= \sup \{ \lambda \in \R \ : \ \exists (\phi,\psi)\geq (0,0), \  (\phi,\psi) \ \text{periodic in} \ x, \ \text{such that} \\
\mathcal{R}(\phi, \psi) +\lambda \phi \leq 0,  
\mathcal{L}(\psi) + \lambda \psi \leq 0, \ \text{and} \ B(\phi, \psi)\leq 0  \}
\end{split}
\end{equation*}
with test functions $(\phi,\psi) \in W_{loc}^{2,3}(\R)\times W_{loc}^{2,3}(\overline{\Omega})$. 

Then, we have:

\begin{proposition}\label{ch1prop:-3}
	Suppose $f$ satisfies \eqref{ch1hyp:per}.
	We have that
	\begin{equation}\label{ch1eq:-3}
	\lambda_p( \R\times(0,r)) \underset{r\to+\infty}{\searrow} \lambda_p ( \Omega).
	\end{equation}
	Moreover, there exists a couple $(u_p, v_p)\in W_{loc}^{2,3}(\R)\times W_{loc}^{2,3}(\overline{\Omega})$ of positive functions periodic in $x$ such that satisfy
	\begin{equation}\label{ch1sys:upvp}
	\left\{
	\begin{array}{ll}
	\mathcal{R}(u_p, v_p)+ \lambda_p ( \Omega)v_p=0, & x\in\R, \\
	\mathcal{L}v_p+ \lambda_p ( \Omega) v_p=0, & (x,y)\in\Omega, \\
	B(u_p, v_p)=0, & x\in\R.
	\end{array}
	\right.
	\end{equation}

\end{proposition}

\begin{proof}
	By inclusion of domains, one has that $\lambda_p( \R\times(0, r))$ is decreasing in $r$.
	Let us call
	\begin{equation*}
	\bar{\lambda}:=\underset{r\to \infty}{\lim} \lambda_p( \R\times(0, r)).
	\end{equation*}
	
	\emph{Step 1}.
	We now want to show that there exists a couple $(\bar{\phi}, \bar{\psi})>(0,0)$, with $\bar{\phi}\in W_{loc}^{2,3}(\R)$ and $\bar{\psi}\in W_{loc}^{2,3}(\overline{\Omega})$, periodic in $x$, that satisfy
	\begin{equation}\label{ch11957}
	\left\{
	\begin{array}{ll}
	\mathcal{R}(\bar{\phi}, \bar{\psi})+ \bar{\lambda} \bar{\phi}=0, & x\in\R, \\
	\mathcal{L}( \bar{\psi})+ \bar{\lambda} \bar{\psi}=0, & (x,y)\in\Omega, \\
	B(\bar{\phi}, \bar{\psi})=0, & x\in\R.
	\end{array}
	\right.
	\end{equation}

	Fix $M>0$.
	First, for all $r>M+2$ consider the periodic eigenfunctions $(\phi_r, \psi_r)$ related to $\lambda_p( \R\times(0,r))$.
	We normalize $(\phi_r, \psi_r)$ so that
	\begin{equation*}
	\phi_r(0)+ \psi_r(0,0)=1.
	\end{equation*}
	
	Then, from the Harnack estimate in Theorem 2.3 of \cite{romain}, there exists $C>0$ such that
	\begin{equation}\label{ch11809}
	\max \{  \underset{I_{M+1}}{\sup} \phi_r, \ \underset{\Omega_{M+1}}{\sup} \psi_r  \} \leq C \min \{  \underset{I_{M+1}}{\inf} \phi_r, \ \underset{\Omega_{M+1}}{\inf} \psi_r  \} \leq C,
	\end{equation}
	where the last inequality comes from the normalization.
	We can use the interior estimate for $\phi_r$  and get
	\begin{equation*}
	|| \phi_r ||_{W^{2,3}(I_M)} \leq C' ( || \phi_r ||_{L^{3}(I_{M+1})}+ || \psi_r ||_{L^{3}(\Omega_{M+1})}   )
	\end{equation*}
	for some $C'$ depending on $M$, $\mu$, $\nu$, and $D$.
	By that and \eqref{ch11809}, we get
	\begin{equation}\label{ch11810}
	|| \phi_r ||_{W^{2,3}(I_M)} \leq C
	\end{equation}
	for a possibly different $C$.
	
	For $\psi_r$, in order to have estimates up to the border $y=0$ of $\Omega_M$, we need to make a construction. Recall that, calling $L:= \mathcal{L}+ \lambda_p( \R\times(0,r))$, $\psi_r$ solves
	\begin{equation*}
	\left\{
	\begin{array}{ll}
	L \psi_r =0, &  (x,y) \in \Omega_{M+1}, \\
	-d \partial_y \psi_r |_{y=0}  + \nu \psi_r|_{y=0}= \mu \phi_r, & x\in I_{M+1}.
	\end{array}
	\right.
	\end{equation*}
	We call
	\begin{equation*}
	\tilde{\psi}_r:= \psi_r e^{-\frac{\nu}{d}y}
	\end{equation*} 
	and the conjugate operator
	\begin{equation*}
	\tilde{L}(w):= e^{-\frac{\nu}{d}y} L\left( e^{\frac{\nu}{d}y} w \right).
	\end{equation*}
	Now, we have
	\begin{equation*}
	\left\{
	\begin{array}{ll}
	\tilde{L}\tilde{\psi}_r =0, &  (x,y) \in \Omega_{M+1}, \\
	-d \partial_y \tilde{\psi}_r |_{y=0}  = \mu \phi_r, & x\in I_{M+1}.
	\end{array}
	\right.
	\end{equation*}
	Next, calling
	\begin{equation}
	w_r(x,y)=\tilde{\psi}_r(x,y)- \frac{d}{\mu} \phi_r(x) y,  
	\end{equation}
	we have that
	\begin{equation}\label{ch11919}
	\left\{
	\begin{array}{ll}
	\tilde{L}w_r = -\dfrac{d}{\mu}  \tilde{L}( \phi_r(x) y), &  (x,y) \in \Omega_{M+1}, \\
	\partial_y {w_r}|_{y=0}  = 0, & x\in I_{M+1}.
	\end{array}
	\right.
	\end{equation}
	Now we define  in the open ball $B_{M+1}$ the function
	\begin{equation}\label{ch11911}
	\bar{w}_r(x,y):=w_r(x, |y|),
	\end{equation}
	that is the extension of $w_r$ by reflection; thanks to the Neumann condition in \eqref{ch11919} and the fact that ${w}_r \in W^{2,3}(\Omega_{M+1})$, we get that $\bar{w}_r \in W^{2,3}(B_{M+1})$. Also, we define the function 
	\begin{equation}\label{ch11912}
	g(x,y)= \frac{d}{\mu}  \tilde{L}( \phi_r(x) |y|).
	\end{equation}
	We also take the operator
	\begin{equation}\label{ch11913}
	\bar{L}w := d \Delta w+ c \partial_x w + 2{\nu} \sigma(y)\partial_y w + \left( f_v(x,0)+ \lambda_p( \R\times(0,r)) + \frac{\nu^2}{d} \right) w 
	\end{equation}
	where $\sigma(y)$ is the sign function given by
	\begin{equation*}
	\sigma(y) := \left\{
	\begin{array}{ll}
	1 & \text{if} \ y\geq 0, \\
	-1 & \text{if} \ y<0.
	\end{array}
	\right.
	\end{equation*}
	Thanks to the definition \eqref{ch11911}, \eqref{ch11912} and \eqref{ch11913}, we get that $\bar{w}_r $ is a weak solution to the equation
	\begin{equation}\label{ch11926}
	- \bar{L} \bar{w}_r = g \quad \text{for} \ (x,y)\in B_{M+1}.
	\end{equation}
	Finally, we can apply the interior estimates and get
	\begin{equation*}
	|| \bar{w}_r  ||_{W^{2,3}(B_M)} \leq C' ( || \bar{w}_r  ||_{L^{\infty}(B_{M+1})}+ || g ||_{L^{3}(B_{M+1})})
	\end{equation*}
	for some $C'$ depending on $M$ and the coefficients of the equation \eqref{ch11926}. But using the definition of $\bar{w}_r $ and the fact that $g$ is controlled by the norm of $\phi_r$,  we get, for a possible different $C'$,
	\begin{equation*}
	|| \bar{w}_r  ||_{W^{2,3}(B_M)} \leq C' ( || \psi_r  ||_{L^{\infty}(\Omega_{M+1})}+|| \phi_r  ||_{L^{\infty}(I_{M+1})}+ || \phi_r ||_{W^{2,3}(I_{M+1})}).
	\end{equation*}
	Using \eqref{ch11810} and \eqref{ch11809},
	we finally have 
	\begin{equation*}
	|| \psi_r  ||_{W^{2,3}(\Omega_M)} \leq C.
	\end{equation*}
	Thanks to that and \eqref{ch11810}, we have that $(\phi_r, \psi_r)$ is uniformly bounded in $W^{2,3}(I_M)\times W^{2,3}(\Omega_M)$ for all $M>0$.
	Hence, up to a diagonal extraction, $(\phi_r, \psi_r)$ converge weakly in $W_{loc}^{2,3}(I_M)\times W_{loc}^{2,3}(\Omega_M)$ to some $(\bar{\phi}, \bar{\psi}) \in W_{loc}^{2,3}(I_M)\times W_{loc}^{2,3}(\Omega_M)$. By Morrey inequality, the convergence is strong in $\mathcal{C}_{loc}^{1, \alpha}(\R)\times \mathcal{C}_{loc}^{1, \alpha}(\overline{\Omega})$ for $\alpha<1/6$. 
	Moreover, $(\bar{\phi}, \bar{\psi})$ are periodic in $x$ since all of the $(\phi_r, \psi_r)$ are periodic.
	Then, taking the limit of the equations in \eqref{ch1sys:r}, we obtain that $(\bar{\phi}, \bar{\psi})$ satisfy \eqref{ch11957}, as wished.

	\emph{Step 2.} We now prove that 
	\begin{equation}\label{ch11402}
	\bar{\lambda} \leq \lambda_p( \Omega).
	\end{equation}
	
	Take $\bar{\lambda}$ and 
	its associated periodic eigenfunctions couple $(\bar{\phi}, \bar{\psi})$ obtained in Step 1. 
	By definition, $\lambda_p( \Omega)$ is the supremum of the set 
	\begin{equation}\label{ch11747}
	\begin{split}
	\mathcal{A}:=  \{ \lambda \in \R \ : \ \exists (\phi,\psi)\geq (0,0), \  (\phi,\psi) \ \text{periodic in} \ x, \ 
	\mathcal{R}(\phi, \psi) +\lambda \phi \leq 0,  \\
	\mathcal{L}(\psi) + \lambda \psi \leq 0, \ \text{and} \ B(\phi, \psi)\leq 0  \}.
	\end{split}
	\end{equation}

	Then, using $(\bar{\phi}, \bar{\psi})$ as test functions, we obtain that $\bar{\lambda}$ is in the set $\mathcal{A}$ given in \eqref{ch11747}. By the fact that $\lambda_p( \Omega)$ is the supremum of $\mathcal{A}$, we get \eqref{ch11402}, as wished.
	
	\emph{Step 3}. We show 
	\begin{equation}\label{ch12006}
	\lambda_p(\Omega) \leq \bar{\lambda}.
	\end{equation}

	Now, take any $\lambda\in\mathcal{A}$ and one of its associate couple $(\phi, \psi)$. Then, by inclusion of domains, one gets that for all $r>0$ it holds 
	\begin{equation*}
	\lambda \leq \lambda_p( \R\times (0,r)).
	\end{equation*}
	Hence, by taking the supremum on the left hand side and the infimum on the right one, we get \eqref{ch12006}. By this and \eqref{ch11402}, equality is proven. Moreover, defining $(u_p, v_p)\equiv(\bar{\phi}, \bar{\psi})$, by \eqref{ch11957}, we have the second statement of the proposition.
\end{proof}

\section{Ordering of the eigenvalues}\label{ch1s:ordering}

This section is dedicated to show some inequalities and relations between the aforementioned eigenvalues.

\subsection{Proof of Theorem \ref{ch1thm:ineq}}

We start by proving Theorem \ref{ch1thm:ineq}.
We stress that this is done for the general setting of $c$ possibly non zero and $f(x,v)$ which may not be periodic.

\begin{proof}[Proof of Theorem \ref{ch1thm:ineq}] \quad
	
	Let us start by proving the first part of the theorem.
	For all $R>0$, there exists $R'>0$ and a point $C\in\R^2$ such that $B_R(C) \subset \Omega_{R'}$: it is sufficient to take $R'=3R$ and $C=(0, \frac{2}{3}R)$. We want to prove that
	\begin{equation}\label{ch11533}
	\lambda_1(-\mathcal{L}, B_R) \geq \lambda_1( \Omega_{R'}).
	\end{equation}
	Suppose by the absurd that \eqref{ch11533} is not true.
	Consider $\psi_R$ the eigenfunction related to $\lambda_1(-\mathcal{L}, B_R)$ and $v_{R'}$ the eigenfunction in the couple $(u_{R'},v_{R'})$ related to $\lambda_1( \Omega_{R'})$. 
	Since $\inf_{B_{R}(C)} v_{R'} >0$, and both eigenfunctions are bounded, there exists
	\begin{equation*}
	\theta^* := \sup \{  \theta\geq 0 \ : \ v_{R'}>\theta \psi_R \ \text{in} \   B_R(C)  \} >0.
	\end{equation*}
	Since $\theta^*$ is a supremum, then there exists $(x^*,y^*)\in \overline{B_R(C)}$ such that $v_{R'}(x^*, y^*)= \theta^* \psi_R (x^*, y^*)$. 
	Then,  $(x^*,y^*)\in {B_R(C)}$ because $v_{R'}>0$ and $\psi_R=0$ in $\partial B_R(C)$.
	Calling $\rho=v_{R'}-\theta^* \psi_R$, in a neighbourhood of $(x^*,y^*)$ we have that
	\begin{equation}\label{ch11602}
	-d \Delta\rho- c \cdot \nabla \rho - f_v(x,0)\rho=\lambda_1(-\mathcal{L}, B_R)\rho + (\lambda_1( \Omega_{R'}) - \lambda_1(-\mathcal{L}, B_R)) v_{R'}.
	\end{equation}
	We know that $\rho(x^*,y^*)=0$ and that $\rho \geq 0$ in $B_R(C)$. Then $(x^*,y^*)$ is a minimum for $\rho$, so $\nabla \rho(x^*,y^*) =0$ and $\Delta \rho (x^*,y^*) \geq 0$.
	Thus, the lefthandside of \eqref{ch11602} is non positive. But by the absurd hypotesis we have $(\lambda_1( \Omega_{R'}) - \lambda_1(-\mathcal{L}, B_R)) v_{R'}>0$. This gives
	\begin{equation*}
	0 \geq -d \Delta\rho(x^*,y^*) = (\lambda_1( \Omega_{R'}) - \lambda_1(-\mathcal{L}, B_R)) v_{R'} (x^*,y^*)>0,
	\end{equation*}
	which
	is a contradiction. With that we obtain that \eqref{ch11533} is true.
	
	Notice that the eigenvalue $\lambda_1(-\mathcal{L}, B_R(C))=\lambda_1(-\mathcal{L}, B_R)$, where $B_R$ is the ball centred in $(0,0)$, because $f(x,v)$ does not depend on $y$, thus system \eqref{ch1sys:L_BR} on $B_R(C)$ and $B_R$ are the same. As a consequence, also their eigenfunctions coincide.
	
	Recall that both $\lambda_1(-\mathcal{L}, \R^2)$ and $\lambda_1( \Omega)$ are limits of eigenvalues on limited domains, by Proposition \ref{ch1prop:lim_B_R}  and Proposition \ref{ch1prop:romain}.
	Now, since for all $R>0$ there exists $R'$ such that \eqref{ch11533} is true, then passing to the limit we find the required inequality.
	
	%\medskip
	%----------------- MODIFICARE ???????????
	%
	%
	%2. We now prove the second statement of the theorem, so from now on we consider $c=0$. Suppose that $\lambda_1(-\mathcal{L}, \R^2)<0$. Then, by the first part of the proof, we have $\lambda_1( \Omega) \leq \lambda_1(-\mathcal{L}, \R^2)<0$. If instead we have $\lambda_1(-\mathcal{L}, \R^2)\geq 0$, by Proposition \ref{ch1prop:ineq} we get that $\lambda_1( \Omega)\geq 0$. This completes the proof of Theorem \ref{ch1thm:ineq}.
	%
	
\end{proof}

\subsection{Further inequalities between the eigenvalues}

In this section, we collect some results on the ordering of periodic and generalised eigenvalues for both system \eqref{ch1sys:fieldroad} and eqaution \eqref{ch1eq:bhroques}.  
Here we require $f$ to be periodic as in \eqref{ch1hyp:per}.

This first result is the analogue of Theorem \eqref{ch1thm:1.7inbr} for the system \eqref{ch1sys:fieldroad}:

\begin{theorem}\label{ch1thm:eq_1_p}
	Suppose $f$ respects hypothesis \eqref{ch1hyp:per}. Then:
	\begin{enumerate}
		\item It holds that $\lambda_1( \Omega)\geq \lambda_p( \Omega)$.
		\item If moreover $c=0$, then we have $\lambda_1( \Omega)= \lambda_p( \Omega)$.
	\end{enumerate}
\end{theorem}

\begin{proof}
	1.
	By definition, $\lambda_p( \Omega)$ is the supremum of the set $\mathcal{A}$ given in \eqref{ch11747},
	while $\lambda_1( \Omega)$ is the supremum of the set
	\begin{equation*}
	\begin{split}
	\{ \lambda \in \R \ : \ \exists (\phi,\psi)\geq (0,0),  \ 
	\mathcal{R}(\phi, \psi) +\lambda \phi \leq 0,  \\
	\mathcal{L}(\psi) + \lambda \psi \leq 0, \ \text{and} \ B(\phi, \psi)\leq 0  \} \supseteq \mathcal{A}.
	\end{split}
	\end{equation*}
	By inclusion of sets, we have the desired inequality.
	
	\medskip
	
	2.
	We call $$\mathcal{H}_R:= H_0^1(I_R)\times H_0^1(\Omega_R \cup (I_R\cup \{0\}) ). $$
	For $(u,v)\in \mathcal{H}_R$, we define
	\begin{equation*}
	Q_R(u,v):= \frac{ \mu \int_{I_R} D |u'|^2 + \nu \int_{\Omega_R} (d|\nabla v|^2-f_v(x,0)v^2) +  \int_{I_R} (\mu u- \nu v|_{y=0})^2      }{\mu \int_{I_R}u^2 + \nu \int_{\Omega_R} v^2}.
	\end{equation*}

	Now we fix $r>0$ and we consider $\lambda_p( \R \times (0,r)  )$ ad its periodic eigenfunctions $(\phi_{r}, \psi_{r})$. We consider $\psi_{r}$ to be extended to $0$ in $\Omega \setminus (\R\times (0,r))$. This way we have $\psi_{r}\in H^1(\Omega_R \cup (I_R\cup \{0\}) )$.
	
	Then for all $R>1$ we choose a $\mathcal{C}^2(\overline{\Omega})$ function $Y_R:\overline{\Omega}\to [0,1]$ such that
	\begin{align*}
	Y_R(x,y)=1 & \qquad  \text{if} \ |(x,y)|<R-1; \\
	Y_R(x,y)=0 & \qquad  \text{if} \ |(x,y)|\geq R; \\ 
	|\nabla Y_R|^2 \leq C; & \hspace{5em} 
	\end{align*}
	where $C$ is a fixed constant independent of $R$. To simplify the notation later, we call $X_R(x):=Y_R(x,y)|_{y=0}$; we also have that $X_R\in\mathcal{C}^2(\R)$ and $|X_R''|\leq C$. We have that
	\begin{equation*}
	(\phi_{r} X_R, \psi_{r} Y_R) \in \mathcal{H}_R.
	\end{equation*}

	Now we want to show that for a suitable diverging sequence $\{R_n\}_{n\in\N}$ we have
	\begin{equation} \label{ch1Claim}
	Q_{R_n} (\phi_{r} X_{R_n}, \psi_{r} Y_{R_n}) \overset{n\to \infty}{\longrightarrow} \lambda_p( \R \times (0,r)  )).
	\end{equation}

	First, let us show a few useful rearrangements of the integrals that define $Q_R (\phi_{r} X_R, \psi_{r} Y_R)$. We have that
	\begin{align*}
	\int_{I_R}|(\phi_{r} X_R)'|^2 &= \int_{I_R} (\phi_{r} X_R)' \, \phi_{r} \,  X_R ' + \int_{I_R} (\phi_{r} X_R)'   \, \phi_{r} ' \, X_R, \\
	& = \int_{I_R} (\phi_{r} X_R)' \, \phi_{r} \,  X_R '  + \left[ (\phi_{r} X_R^2)  \, \phi_{r} ' \right]_{-R}^R-\int_{I_R}  (\phi_{r} X_R)  \, \left( \phi_{r} '' \, X_R + \phi_{r} ' \,  X_R' \right), \\
	&= \int_{I_R} \phi_{r}^2 \, |X_R '|^2 +  \left[ (\phi_{r} X_R^2)  \, \phi_{r} ' \right]_{-R}^R -\int_{I_R}  \phi_{r} '' \, \phi_{r} \, X_R^2 ,
	\end{align*}
	by having applied integration by parts on the second line and trivial computation in the others.
	Since $X_R(R)= X_R(-R)=0$ and $X_R '$ is supported only in $I_R \setminus I_{R-1}$, we get 
	\begin{equation}\label{ch1eq:parte2}
	\mu D\int_{I_R}|(\phi_{r} X_R)'|^2 = -\mu D\int_{I_R}  \phi_{r} '' \, \phi_{r} \, X_R^2  + \mu D \int_{I_R \setminus I_{R-1}} \phi_{r}^2 \, |X_R '|^2.
	\end{equation}	
	With similar computations we get 
	\begin{equation}\label{ch1eq:parte1}
	\int_{\Omega_R} d|\nabla (\psi_{r} \, Y_R)|^2 = - \int_{\Omega_R} d\Delta \psi_{r} \, \psi_{r} \, Y_R^2 - \int_{I_R} (d\partial_y \psi_{r}) \psi_{r} \, X_R^2 + \int_{\Omega_R \setminus {\Omega_{R-1} }} d|\nabla Y_R|^2 \psi_{r}^2.
	\end{equation}
	Then, we also have 
	\begin{equation} \label{ch1eq:parte3}
	\int_{I_R} (\mu \phi_{r} \, X_R - \nu \psi_{r} \, X_R)^2 = \int_{I_R} \mu \phi_{r} \, X_R^2 (\mu \phi_{r}- \nu \psi_{r}) - \int_{I_R} \nu \psi_{r} \, X_R^2 (\mu \phi_{r}- \nu \psi_{r}).
	\end{equation}
	We now recall that $(\phi_{r}, \psi_{r})$ is an eigenfunction for the problem \eqref{ch1sys:r}.
	Thanks to the third equation of \eqref{ch1sys:r}, the second term in \eqref{ch1eq:parte1} cancel out with the second term in \eqref{ch1eq:parte3}. Moreover we can sum the first term of \eqref{ch1eq:parte2} and the first term of \eqref{ch1eq:parte3} and get
	\begin{equation*}
	-\int_{I_R} \mu D \phi_{r} '' \, \phi_{r} \, X_R^2 + \int_{I_R} \mu \phi_{r} \, X_R^2 (\mu \phi_{r} - \nu \psi_{r}) = \int_{I_R} \mu \lambda_p( \R\times (0, r) ) \phi_{r}^2 \, X_R^2.
	\end{equation*}
	Moreover we have that 
	\begin{equation*}
	- \int_{\Omega_R} d\Delta \psi_{r} \, \psi_{r} \, Y_R^2 - \int_{\Omega_R} f_v(x,0) \psi_{r}^2 \, Y_R^2=  \int_{\Omega_R}  \lambda_p( \R\times (0, r) ) \psi_{r}^2 \, Y_R^2  .
	\end{equation*}
	So, if we call
	\begin{equation*}
	P_R := \frac{ \mu  \int_{I_R \setminus I_{R-1}} D\phi_{r}^2 \, |X_R '|^2 + \nu \int_{\Omega_R \setminus {\Omega_{R-1} }} d|\nabla Y_R|^2 \psi_{r}^2}{\mu \int_{I_R}(\phi_{r} X_R)^2 + \nu \int_{\Omega_R} (\psi_{r} Y_R)^2},
	\end{equation*}
	we have that 
	\begin{equation*}\label{ch10014}
	Q_R (\phi_{r} X_R, \psi_{r} Y_R) = \lambda_p( \R\times (0, r) ) + P_R.
	\end{equation*}
	Proving \eqref{ch1Claim} is equivalent to show that
	\begin{equation} \label{ch11604}
	P_{R_n} \overset{n\to \infty}{\longrightarrow} 0
	\end{equation} 
	for some diverging sequence $\{R_n\}_{n\in \N}$.
	Suppose by the absurd \eqref{ch11604} is not true. 
	First, by the fact that the derivatives of $X_R$ and $Y_R$ are bounded, for some positive constant $C$ we have that 
	\begin{equation*}
	0 \leq P_R \leq C \frac{ \mu  \int_{I_R \setminus I_{R-1}} \phi_{r}^2  + \nu \int_{\Omega_R \setminus {\Omega_{R-1} }} \psi_{r}^2}{\mu \int_{I_R}(\phi_{r} X_R)^2 + \nu \int_{\Omega_R} (\psi_{r} Y_R)^2}
	\end{equation*}
	By the absurd hypothesis,  we have that
	\begin{equation} \label{ch11652}
	\underset{R\to \infty}{\liminf} \, P_R = \xi >0.
	\end{equation}
	Now let us define for all $R\in \N$ the quantity
	\begin{equation*}
	\alpha_R:= \mu \int_{I_R\setminus I_{R-1}}\phi_{r} ^2 + \nu \int_{\Omega_R \setminus \Omega_{R-1}} \psi_{r}^2.
	\end{equation*}
	Since $\phi_r$ and $\psi_r$ are bounded from above, we have that for some constant $k$ depending on $r$, $\mu$, and $\nu$, we have
	\begin{equation}\label{ch1H}
	\alpha_R \leq k R.
	\end{equation}
	For $R\in \N$ one has
	\begin{equation*}
	\mu \int_{I_R}(\phi_{r} X_R)^2 + \nu \int_{\Omega_R} (\psi_{r} Y_R)^2 = \sum_{n=1}^{R-1} \alpha_n + \mu \int_{I_R \setminus I_{R-1}}(\phi_{r} X_R)^2 + \nu \int_{\Omega_R \setminus \Omega_{R-1}} (\psi_{r} Y_R)^2.
	\end{equation*}
	By comparison with \eqref{ch11652}, we have 
	\begin{equation*}
	\underset{R\to \infty}{\liminf} \, \frac{\alpha_R}{\sum_{n=1}^{R-1} \alpha_n} \geq  \underset{R\to \infty}{\liminf} \, \frac{ \alpha_R}{\sum_{n=1}^{R-1} \alpha_n + \mu \int_{I_R \setminus I_{R-1}}(\phi_{r} X_R)^2 + \nu \int_{\Omega_R \setminus \Omega_{R-1}} (\psi_{r} Y_R)^2} \geq \frac{\xi}{C},
	\end{equation*}
	so for $0<\varepsilon< \xi /C$ we have
	\begin{equation}\label{ch1G}
	\alpha_R > \varepsilon \sum_{n=1}^{R-1} \alpha_n
	\end{equation}
	Thanks to \eqref{ch1G} we perform now a chain of inequalities:
	\begin{equation*}
	\alpha_{R+1} > \varepsilon \sum_{n=1}^{R} \alpha_n = \varepsilon \left( \alpha_R + \sum_{n=1}^{R-1} \alpha_n \right) > \varepsilon(1+\varepsilon)\sum_{n=1}^{R-1} \alpha_n > \dots > (1+\varepsilon)^{R+1} \frac{\varepsilon \alpha_1}{(1+\varepsilon)^3} .
	\end{equation*}
	from with we derive that $\alpha_{R}$ diverges as an exponential, in contradiction with the inequality in \eqref{ch1H}. Hence we obtain that 
	\eqref{ch11604} is true, so \eqref{ch1Claim} is also valid.

	By Proposition 4.5 in \cite{romain}, we have that
	\begin{equation}\label{ch11207}
	\lambda_1( \Omega_R) = \underset{  \substack{(u,v)\in \mathcal{H}_R, \\ (u,v)\neq (0,0)}   }{\min} Q_R(u,v).
	\end{equation}
	Hence by \eqref{ch11207} we have that
	\begin{equation*}
	\lambda_1( \Omega_R) \leq Q_R (\phi_{r} X_R, \psi_{r} Y_R).
	\end{equation*}

	Since for all $r>0$ there exist $R>0$ so that \eqref{ch1Claim} holds, we have moreover that
	\begin{equation*}
	\lambda_1( \Omega) \leq \lambda_p( \R\times (0, r) ).
	\end{equation*}
	Then, recalling Proposition \ref{ch1prop:-3}, we get that 
	\begin{equation*}
	\lambda_1( \Omega) \leq \lambda_p( \Omega ).
	\end{equation*}
	Since the reverse inequality was already stated in the first part of this theorem, one has the thesis.
\end{proof}

At last, we prove this proposition of the bounds for $\lambda_p(-\mathcal{L},\Omega)$.

\begin{proposition}\label{ch1prop:Lp_Omega}
	Suppose $f$ satisfies \eqref{ch1hyp:per}.
	We have that
	\begin{equation*}
	\lambda_p(-\mathcal{L}, \R^2)
	\leq \lambda_p(-\mathcal{L}, \Omega) \leq 
	\lambda_1(-\mathcal{L}, \Omega) 
	\end{equation*}
	and if $c=0$ the equality holds.
\end{proposition}

\begin{proof}
	Consider any $r>0$ and take $\lambda_p(-\mathcal{L}, \R\times(-r,r))$ and its eigenfunction $\psi_r$ solving \eqref{ch1sys:bary2}, that is periodic in $x$.
	Then take $\lambda_p(-\mathcal{L}, \R^2)$ and its periodic eigenfunction $\psi_p$, that as we have seen in \eqref{ch1eq:-5} does not depend on $y$, therefore it is limited and has positive infimum, and solves \eqref{ch1sys:L_R2_p}. Then, $\lambda_p(-\mathcal{L}, \R\times(-r,r))$ and $\lambda_p(-\mathcal{L}, \R^2)$ are eigenvalues for the same equation in two domains with one containing the other; hence, one gets that 
	\begin{equation}\label{ch11432}
	\lambda_p(-\mathcal{L}, \R^2)\leq \lambda_p(-\mathcal{L}, \R\times(-r,r)).
	\end{equation}
	By using \eqref{ch1eq:-2}, from \eqref{ch11432} we have
	\begin{equation}\label{ch11429}
	\lambda_p(-\mathcal{L}, \R^2)\leq \lambda_p(-\mathcal{L}, \Omega).
	\end{equation}
	
	Given $R<r$, we can repeat the same argument for $\lambda_1(-\mathcal{L}, B_R)$ and $\lambda_p(-\mathcal{L}, \R\times(-r,r))$  and get
	\begin{equation}\label{ch11433}
	\lambda_p(-\mathcal{L}, \R\times(-r,r)) \leq \lambda_1(-\mathcal{L}, B_R).
	\end{equation}
	By \eqref{ch1eq:-2} and Proposition \ref{ch1prop:lim_B_R}, we get
	\begin{equation*}
	\lambda_p(-\mathcal{L}, \Omega) \leq \lambda_1(-\mathcal{L}, \Omega).
	\end{equation*}
	This and \eqref{ch11429} give the first statement of the proposition.
	
	If $c=0$, by the second part of Theorem \ref{ch1thm:1.7inbr} we get that $\lambda_p(-\mathcal{L}, \R^2) =\lambda_1(-\mathcal{L}, \Omega)$, hence we have
	\begin{equation*}
	\lambda_p(-\mathcal{L}, \R^2)= \lambda_p(-\mathcal{L}, \Omega) =\lambda_1(-\mathcal{L}, \Omega),
	\end{equation*}
	as wished.
\end{proof}

\subsection{Proof of Theorem \ref{ch1thm:comparison}}

Owing Theorems \ref{ch1thm:char} and \ref{ch1thm:2.6inbhroques} together with the estimates on the eigenvalues proved in the last subsection, we are ready to prove Theorem \ref{ch1thm:comparison}.

\begin{proof}[Proof of Theorem \ref{ch1thm:comparison}]
	By Theorem \ref{ch1thm:1.7inbr}, we have that $\lambda_1(-\mathcal{L}, \R^2)=\lambda_p(-\mathcal{L}, \R^2)$. Then, by Corollary \ref{ch1thm:ineq2}, if $\lambda_1(\Omega)<0$ then
	$\lambda_p(-\mathcal{L}, \R^2)<0$, and if $\lambda_1(\Omega)\geq 0$ then
	$\lambda_p(-\mathcal{L}, \R^2)\geq 0$.
	
	Observe that, when $c=0$, choosing $N=2$ and $L=(\ell, 0)$, the operator $\mathcal{L'}$ defined in \eqref{ch1def:mathcal_L'} coincides with $\mathcal{L}$.
	Then, the affirmations on the asymptotic behaviour of the solutions of the system with and without the road comes from the characterisations in Theorem \ref{ch1thm:char} and \ref{ch1thm:2.6inbhroques}.
	
\end{proof}

\section{Large time behaviour for a periodic medium and $c=0$}\label{ch1s:lb}

We start considering the long time behaviour of the solutions. As already stated in Theorem \ref{ch1thm:char}, the two possibilities for a population evolving through \eqref{ch1sys:fieldroad} are persistence and extinction. We treat these two case in separate sections.

Before starting our analysis, we recall a comparison principle  first appeared in \cite{brr} that is fundamental for treating 
system \eqref{ch1sys:fieldroad}. We recall that a generalised subsolution (respectively, supersolution) is the supremum (resp. infimum) of two subsolutions (resp. supersolutions).

\begin{proposition}[Proposition 3.2 of \cite{brr}]\label{ch1prop:comparison}
	Let $(\underline{u}, \underline{v})$ and $(\overline{u}, \overline{v})$ be respectively a generalised subsolution bounded from
	above and a generalised supersolution bounded from below of \eqref{ch1sys:fieldroad} satisfying $\underline{u} \leq \overline{u}$ and $\underline{v} \leq \overline{v}$
	at $t = 0$. Then, either $\underline{u} \leq \overline{u}$ and $\underline{v} \leq \overline{v}$ for all $t$, or there exists $T > 0$ such that
	$(\underline{u}, \underline{v}) \equiv (\overline{u}, \overline{v})$ for $t\leq T$.
\end{proposition}

The original proof is given for the case of $f$ homogeneous in space; however, it can be adapted with changes so small that we find it useless to repeat it.

Proposition \ref{ch1prop:comparison} gives us important informations on the behaviour at microscopic level. In fact, it asserts that if two pairs of population densities are ``ordered'' at an initial time, then the order is preserved during the evolution according to the equations in \eqref{ch1sys:fieldroad}.

\subsection{Persistence}\label{ch1ss:persistence}

The aim of this section is to prove the second part of Theorem \ref{ch1thm:char}.
First, we are going to show a Liouville type result, that is Theorem \ref{ch1lemma:pbss}, and then we will use that to derive the suited convergence.

We start with some technical lemmas.

\begin{lemma}\label{ch1lemma:converg}
	Let $(u,v)$ be a bounded stationary solution to \eqref{ch1sys:fieldroad} and let $\{(x_n, y_n) \}_{n\in\N}\subset \Omega$ be a sequence of points such that $\{ x_n\}_{n\in\N}$ modulo $\ell$ tends to some $x'\in[0,\ell]$. 
	Then:
	\begin{enumerate}
		\item if $\{ y_n\}_{n\in\N}$ is bounded,
		the sequence of function $\{(u_n, v_n)  \}_{n\in\N}$ defined as
		\begin{equation}\label{ch11648}
		(u_n(x), v_n(x, y))=(u(x+x_n), v(x+x_n, y))
		\end{equation}
		converges  up to a subsequence  to $(\tilde{u}, \tilde{v})$ in $\mathcal{C}_{loc}^2(\R\times\Omega)$ and $(\tilde{u}(x-x'), \tilde{v}(x-x',y)$ is a bounded stationary solution to \eqref{ch1sys:fieldroad}.
		\item if $\{ y_n\}_{n\in\N}$ is unbounded,
		the sequence of function $\{ v_n \}_{n\in\N}$ defined as
		\begin{equation}\label{ch11649}
		v_n(x, y)= v(x+x_n, y+y_n)
		\end{equation}
		converges  up to a subsequence  to $ \tilde{v}$ and $\tilde{v}(x-x', y)$  in $\mathcal{C}_{loc}^2(\R^2)$ is a bounded stationary solution to the second equation in  \eqref{ch1sys:fieldroad} in $\R^2$.
	\end{enumerate}
\end{lemma}

\begin{proof}
	
	Let us call $V=\max\{ \sup u, \sup v  \}$.   
	For all $n\in\N$, there exists $x_n'\in[0,\ell)$ such that $x_n-x_n'\in\ell \Z$.
	
	We start with the case of bounded $\{y_n\}_{n\in\N}$.
	By the periodicity of $f$,
	we have that $(u_n, v_n)$ defined in \eqref{ch11648} is a solution to
	\begin{equation*}
	\left\{
	\begin{array}{lr}
	-D  u '' -c  u' - \nu  v|_{y=0} + \mu u= 0,   & x\in \R,  \\
	v -d \Delta v -c \partial_x v  =f(x+ x_n',v),  & (x, y)\in \Omega, \\
	-d  \partial_y{v}|_{y=0} + \nu v|_{y=0} -\mu u=0, & x\in\R,
	\end{array}
	\right.
	\end{equation*}
	
	Fix $p\geq 1$ and three numbers $j>h>k>0$; we use 
	the notation in \eqref{ch11722} for the sets $I_R$ and $\Omega_R$ for $R= k,\ h, \ j$.
	By Agmon-Douglis-Nirenberg estimates (see for example Theorem 9.11 in \cite{GT}), we have 
	\begin{equation*}
	\norm{u_n}_{W^{2,p}( I_h)} \leq C \left(  \norm{u_n}_{L^p(  I_j)} +\norm{v_n(x,0)}_{L^p(  I_j)}     \right).
	\end{equation*}
	To find the same estimate for the norm of $v_n$, we have to make the same construction used in the proof of Proposition \ref{ch1prop:-3} to find the bound for $\psi_r$. In the same way, we get
	\begin{equation*}
	\begin{split}
	\norm{v_n}_{W^{2,p}(  \Omega_h)} \leq C \Big(  \norm{u_n}_{L^p(  I_j)}   +\norm{v_n}_{L^p(  \Omega_j)} + \norm{f}_{L^p( I_j \times (0, V) )}     \Big) .
	\end{split}
	\end{equation*}
	where the constant $C$, possibly varying in each inequality, depends on $\nu$, $\mu$, $d$, $D$, $h$ and $j$. 
	Using the boundedness of $u$ and $v$, for a possible different $C$ depending on $f$ we get
	\begin{align*}
	\norm{u_n}_{W^{2,p}( I_h)} &\leq C V, \\
	\norm{v_n}_{W^{2,p}(  \Omega_h)} &\leq C V.
	\end{align*}

	Then, we apply the general Sobolev inequalities (see \cite{evans}, Theorem 6 in 5.6) and get for some $\alpha$ depending on $p$, that
	\begin{align*}
	\norm{u_n}_{\mathcal{C}^{\alpha}(  I_h)} & \leq C \norm{u_n}_{W^{2,p}(  I_h)}  \leq CV, \\
	\norm{v_n}_{\mathcal{C}^{\alpha}(  \Omega_h)} &\leq C \norm{v_n}_{W^{2,p}(   \Omega_h)}  \leq CV.
	\end{align*}

	Now we can apply Schauder interior estimates for the oblique boundary condition (see for example Theorem 6.30 in \cite{GT}) and find that 
	\begin{align*}
	\norm{u_n}_{\mathcal{C}^{2,\alpha}(I_k)} &\leq C \left(  \norm{u_n}_{\mathcal{C}^{\alpha}(I_h)} +\norm{v_n(x,0)}_{\mathcal{C}^{\alpha}(I_h)}     \right) \leq CV, \\
	\norm{v_n}_{\mathcal{C}^{2,\alpha}(\Omega_k)} &\leq C \Big(  \norm{u_n}_{\mathcal{C}^{\alpha}(I_h)} 
	+\norm{v_n}_{\mathcal{C}^{\alpha}(\Omega_h)} + \norm{f}_{\mathcal{C}^{\alpha}(I_h \times[0,V])}     \Big) \leq CV.
	\end{align*}
	So the sequences $\{u_n\}_{n\in\N}$ and $\{v_n\}_{n\in\N}$ are bounded locally in space in $C^{2,\alpha}$. By compactness we can extract converging subsequences with limits $\tilde{u}(x)$ and $\tilde{v}(x,y)$. Moreover, since by hypothesis $x_n'\to x'$ as $n\to+\infty$, we have that $(\tilde{u}, \tilde{v})$ is a solution
	\begin{equation*}
	\left\{
	\begin{array}{lr}
	-D  u '' -c  u' - \nu  v|_{y=0} + \mu u= 0,   & x\in \R,  \\
	v -d \Delta v -c \partial_x v  =f(x+ x',v),  & (x, y)\in \Omega, \\
	-d  \partial_y{v}|_{y=0} + \nu v|_{y=0} -\mu u=0, & x\in\R,
	\end{array}
	\right.
	\end{equation*}
	This concludes the proof of the first statement.
	
	Now suppose that $\{  y_n  \}_{n\in\N}$ is unbounded and, up to a subsequence, we can suppose that 
	\begin{equation}\label{ch11827}
	y_n \overset{n\to\infty}{\longrightarrow} +\infty.
	\end{equation}
	Then, 
	the function defined in \eqref{ch11649} solves the equation 
	\begin{equation*}
	-d\Delta v_n -c \partial_x v_n = f(x+x_n', v) \quad \text{for} \ (x,y)\in\R\times(-y_n,0) 
	\end{equation*}
	with the boundary condition $-d\partial_y v_n(x, y_n)  + \nu v_n(x, -y_n)- \mu u(x+x_n)=0$. 
	Fix $p\geq 1$ and three numbers $j>h>k>0$; we denote by $B_R$ the open ball of $\R^2$ centred in $(0,0)$ and with radius $R$, and we will consider $R=j, \ h, \ k$. Notice that by \eqref{ch11827} there exists $N\in\N$ we have that $y_n>j$ for all $n\geq N$.
	Hence, applying the previous estimates to $v_n$ for all $n\geq N$, we find that
	\begin{equation*}
	\begin{split}
	\norm{v_n}_{W^{2,p}(  B_h)} \leq C \Big(  \norm{v_n}_{L^p(  B_j)} + \norm{f}_{L^p( I_j \times (0, V) )}     \Big) \leq CV
	\end{split}
	\end{equation*}
	and then that
	\begin{equation*}
	\norm{v_n}_{\mathcal{C}^{2,\alpha}(B_k)} \leq C \Big(  \norm{v_n}_{\mathcal{C}^{\alpha}(B_h)} + \norm{f}_{\mathcal{C}^{\alpha}(I_h \times[0,V])}     \Big) \leq CV.
	\end{equation*}
	So the sequence $\{v_n\}_{n\in\N}$ is bounded locally in space in $C^{2,\alpha}(\R^2)$ and by compactness we can extract converging subsequence with limit $\tilde{v}(x,y)$, that satisfy
	\begin{equation*}
	-d\Delta v_n -c \partial_x v_n = f(x+x', v) \quad \text{for} \ (x,y)\in\R^2,
	\end{equation*}
	which gives the claim.
	
\end{proof}

The second lemma is similar to the first one, but treats a shifting in time.

\begin{lemma}\label{ch1lemma:mono}
	Let $(u,v)$ be a bounded solution to \eqref{ch1sys:fieldroad} which is monotone in time and let $\{ t_n\}_{n\in\N}\subset \R_{\geq 0}$ be a diverging sequence. Then, the sequence $\{(u_n, v_n)\}_{n\in \N}$ defined by
	\begin{equation}\label{ch11840}
	(u_n(t,x), v_n(t,x, y))=(u(t+t_n,x), v(t+t_n,x, y))
	\end{equation}
	converges in $C_{loc}^{1,2,\alpha}$ to a couple of functions $(\tilde{u}, \tilde{v})$ which is a stationary solution to \eqref{ch1sys:fieldroad}.
\end{lemma}

\begin{proof}
	We call $V=\max \{ \sup u, \sup v   \}$.
	For every fixed $x\in \R$  we have that $u_n(t,x)$ is an monotone bounded sequence. Then, we can define a function $\tilde{u}(x)$ as 
	\begin{equation}\label{ch11720}
	\tilde{u}(x) = \underset{n\to\infty}{\lim} {u_n(t,x)}
	\end{equation}
	and $0\leq \tilde{u}(x)\leq U$. 
	Analogously, for all $(x,y)\in \Omega$  we can define
	\begin{equation}\label{ch11721}
	\tilde{v}(x,y) = \underset{n\to\infty}{\lim} {v_n(t,x,y)}
	\end{equation}
	and $0 \leq \tilde{v}(x,y)\leq V$. 
	
	Fix $p\geq 1$, $T>0$ and three numbers $k<h<j$; we use 
	the notation in \eqref{ch11722} for the sets $I_R$ and $\Omega_R$ for $R= k,\ h, \ j$. 
	For $S$ an open subset in $\R^N$, in this proof we denote the space of function with one weak derivative in time and two weak derivatives in space by $W_p^{1,2}(S)$. 
	By Agmon-Douglis-Nirenberg estimates we have 
	\begin{equation*}
	\norm{u_n}_{W^{1,2}_p(  I_h)} \leq C \left(  \norm{u_n}_{L^p((0,T)\times I_j)} +\norm{v_n(t,x,0)}_{L^p((0,T)\times I_j)}     \right) \leq CV.
	\end{equation*}
	To find the same estimate for the norm of $v_n$, we have to make the same construction used in the proof of Proposition \ref{ch1prop:-3} to find the bound for $\psi_r$. In the same way, we get
	\begin{equation*}
	\begin{split}
	\norm{v_n}_{W^{1,2}_p((0,T)\times \Omega_h)} \leq C \Big(  \norm{u_n}_{L^p((0,T)\times I_j)} %+  \hspace{5em}\\  
	+\norm{v_n}_{L^p((0,T)\times \Omega_j)} + \norm{f}_{L^p( I_j \times (0, V) )}     \Big) \leq CV.
	\end{split}
	\end{equation*}
	where the constant $C$, possibly varying in each inequality, depends on $\nu$, $\mu$, $d$, $D$, $T$, $h$ and $j$. Then, we apply the general Sobolev inequalities (see \cite{evans}, Theorem 6 in 5.6) and get for some $\alpha$ depending on $p$, that
	\begin{align*}
	\norm{u_n}_{\mathcal{C}^{\alpha}((0,T)\times I_h)} & \leq C \norm{u_n}_{W^{1,2}_p((0,T)\times I_h)}  \leq CV, \\
	\norm{v_n}_{\mathcal{C}^{\alpha}((0,T)\times \Omega_h)} &\leq C \norm{v_n}_{W^{1,2}_p((0,T)\times  \Omega_h)}  \leq CV.
	\end{align*}
	Moreover, since for $n\in \N$ the functions $u_n$ and $v_n$ are just time translation of the same functions $\tilde{u}$ and $\tilde{v}$, we also have that
	\begin{align*}
	\norm{u_n}_{\mathcal{C}^{\alpha}((0,+\infty)\times I_h)} &\leq CV, \\
	\norm{v_n}_{\mathcal{C}^{\alpha}((0,+\infty)\times  \Omega_h)} &\leq CV.
	\end{align*} 
	Now we can apply Schauder interior estimates (see for example Theorem 10.1 in Chapter IV of \cite{ladyzhenskaya}) and find that 
	\begin{align*}
	\norm{u_n}_{\mathcal{C}^{1,2,\alpha}((0,+\infty)\times I_k)} &\leq C \left(  \norm{u_n}_{\mathcal{C}^{\alpha}((0,+\infty)\times I_h)} +\norm{v_n(t,x,0)}_{\mathcal{C}^{\alpha}((0,+\infty)\times I_h)}     \right) \leq CV, \\
	\norm{v_n}_{\mathcal{C}^{1,2,\alpha}((0,+\infty)\times \Omega_k)} &\leq C \Big(  \norm{u_n}_{\mathcal{C}^{\alpha}((0,+\infty)\times I_h)} + \\ &\hspace{5em} 
	+\norm{v_n}_{\mathcal{C}^{\alpha}((0,+\infty)\times \Omega_h)} + \norm{f}_{\mathcal{C}^{\alpha}(I_h \times[0,V])}     \Big) \leq CV.
	\end{align*}
	So the sequences $\{u_n\}_{n\in\N}$ and $\{v_n\}_{n\in\N}$ are bounded locally in space in $C^{1,2,\alpha}$. By compactness we can extract converging subsequences with limits $q(t,x)$ and $p(t,x,y)$ that satisfy system \eqref{ch1sys:fieldroad}. But as said in \eqref{ch11720} and \eqref{ch11721} the sequences
	$\{u_n\}$ and $\{v_n\}$ also converge punctually to $\tilde{u}$ and $\tilde{v}$, that are stationary functions. Then, the couple $(\tilde{u}, \tilde{v})$ is a positive bounded stationary solution of system \eqref{ch1sys:fieldroad}.
\end{proof}

The following lemma gives essentials information on the stationary solutions, on which the uniqueness result of Theorem \ref{ch1lemma:pbss} will rely on.

\begin{lemma}\label{ch1lemma:pos_inf}
	Suppose that $c=0$, $f$ satisfies \eqref{ch1hyp:0}-\eqref{ch1hyp:per} and that
	$\lambda_1( \Omega)<0$. Then, every stationary bounded solution $(u, v)\not\equiv(0,0)$ of system \eqref{ch1sys:fieldroad} respects
	\begin{equation*}
	\underset{\R}{\inf} \, u >0, \quad \underset{\Omega}{\inf} v>0.
	\end{equation*}
\end{lemma}

\begin{proof}
	
	\emph{Step 1: sliding in $x$.}
	If $\lambda_1( \Omega)<0$, thanks to Proposition \ref{ch1prop:romain} there exists $R>0$ such that $\lambda_1( \Omega_R)<0$. 
	Since $\lambda_1( \Omega_R)$ is monotonically decreasing in $R$, 
	we can suppose that $R> \ell$.
	By a slight abuse of notation, let us call $(u_R,v_R)$ the eigenfunctions associated with $\lambda_1( \Omega_R)<0$ extended to 0 in $\R \times \Omega \setminus (I_R\times \Omega_R)$.
	
	We claim that there exists $\varepsilon>0$ such that $\varepsilon(u_R,v_R)$ is a subsolution for system \eqref{ch1sys:fieldroad}.
	In fact, we have that
	\begin{equation*}
	\underset{v\to 0^+}{\lim} \dfrac{f(x,v)}{v} = f_v(x,0),
	\end{equation*}
	so for $\varepsilon$ small enough we have that
	\begin{equation}\label{ch12220}
	\dfrac{f(x,\varepsilon v_R)}{\varepsilon v_R} > f_v(x,0) + \lambda_1( \Omega_R).
	\end{equation}
	Then, 
	\begin{equation}\label{ch12221}
	\left\{
	\begin{array}{ll}
	-D \varepsilon u_R '' -c \varepsilon  u_R' - \nu \varepsilon v_R|_{y=0} +\mu \varepsilon u_R=\lambda_1( \Omega_R ) \varepsilon u_R \leq 0,   & x\in I_R,  \\
	-d \Delta \varepsilon v_R -c \partial_x \varepsilon v_R  =(f_v(x,0)+ \lambda_1( \Omega_R) \varepsilon v_R \leq f(x, \varepsilon v_R),  & (x, y)\in \Omega_R, \\
	-d  \varepsilon\partial_y{v_R}|_{y=0} + \nu \varepsilon v_R|_{y=0} -\varepsilon u_R=0, & x\in I_R,
	\end{array}
	\right.
	\end{equation}
	so $\varepsilon(u_R,v_R)$ is a subsolution. 
	
	Decreasing $\varepsilon$ if necessary, we have that $\varepsilon(u_R,v_R)<(u,v)$ because $u$ and $v$ are strictly positive in all points of the domain while $(u_R, v_R)$ has compact support. Now we translate $\varepsilon(u_R,v_R)$ in the variable $x$ by multiples of $\ell$; given $k\in \Z$, we call
	\begin{align*}
	u_{R, k}(x):= \varepsilon u_R(x-k\ell), \quad  & I_{R,k}=(k\ell-R,k\ell+R), \\
	v_{R, k}(x,y):=\varepsilon v_R(x-k\ell,y), \quad  & \Omega_{R,k}= B_R(k\ell, 0) \cap \Omega.
	\end{align*}
	The couple $(u_{R,k}, v_{R,k})$ is still a subsolution to system \eqref{ch1sys:fieldroad} because is a translation of a subsolution by multiple of the periodicity of the coefficients in the equations.
	Suppose by the absurd that there exists $k\in \Z$ such that $(u_{R,k}, v_{R,k})\not < (u,v)$. 
	Since $u$ and $v$ are strictly positive in all points of respectively $\R$ and $\Omega$, while $u_{R,k}$ and $v_{R,k}$ have compact support, by decreasing $\varepsilon$ if necessary, we have that $(u_{R,k}, v_{R,k})\leq (u,v)$ and either  there exists $\bar{x}\in I_{R,k}$ such that $ u_{R,k}(\bar{x})=u(\bar{x})$ or there exists $(\bar{x}, \bar{y})\in \overline{\Omega}_{R,k} $ such that $v_{R,k}(\bar{x}, \bar{y})=v(\bar{x}, \bar{y})$. 
	Then, by the Comparison Principle, we have that $(u_{R,k}, v_{R,k})\equiv (u,v)$, which is absurd because $u_{R,k}$ and $v_{R,k}$ are compactly supported.
	Therefore, we have
	\begin{equation}\label{ch11811}
	\begin{array}{rl}
	u(x) > \varepsilon u_R(x+k\ell), &\quad \forall x\in\R, \ \forall k\in\Z, \\
	v(x,y) > \varepsilon v_R(x+k\ell,y), &\quad \forall (x,y)\in\Omega,  \ \forall k\in\Z.
	\end{array}
	\end{equation}
	Fix $Y<\sqrt{R^2-\ell^2}$. Then, let us call
	\begin{equation*}
	\delta_Y := \min\{ \underset{[0, \ell]}{\min} \, \varepsilon u_R(x),  \underset{[0,\ell]\times[0,Y] }{\min} \varepsilon v_R(x,y) \}. 
	\end{equation*}
	Since $[0,\ell]\times(0,Y) \subset \Omega_R$ and $[0,\ell]\subset I_R$,  we have that $\delta_Y>0$.
	Then,  \eqref{ch11811} implies that 
	\begin{equation}\label{ch11749}
	\begin{array}{ll}
	u(x)>\delta_Y, & \text{for} \ x\in \R, \\
	v(x,y)>\delta_Y, & \text{for} \ x\in \R, \ y\in[0,Y].
	\end{array}
	\end{equation}

	\emph{Step 2: sliding in $y$.}
	Recall that by Corollary \ref{ch1thm:ineq2} we have $\lambda_1( \Omega)<0$ implies $\lambda_1(-\mathcal{L},\Omega) <0$ and by Proposition \eqref{ch1prop:-3} it holds  $\lambda_p(-\mathcal{L}, \R^2)\leq \lambda_1(-\mathcal{L},\Omega)<0$. By Proposition \ref{ch1prop:Lp_Omega}
	and by \eqref{ch1eq:-2} we have that for some $r>0$ we have $\lambda_p(-\mathcal{L},\R\times (-r,r))<0$. Then, let us call $v_r$ the eigenfunction related to $\lambda_p(-\mathcal{L},\R\times  (-r,r))$ extended to 0 outside its support; repeating the classic argument, one has that for some $\theta>0$ we have $\theta v_r$ extended to 0 outside $\R\times  (-r,r)$ is a subsolution for the second equation in system \eqref{ch1sys:fieldroad}.
	For all $h>0$, let us now call $\varphi_h(x,y):=v_r(x, y+h)$.
	Since $v_r$ is periodic in the variable $x$, we have that $v_r$ is uniformly bounded. 
	Now take $Y>2r$ and $h_0>r$ such that $Y>h_0+r$; by decreasing $\theta$ if necessary, we get that $\theta v_r < \delta_Y$. 
	Hence, we get 
	\begin{equation}\label{ch11756}
	\theta \varphi_{h_0}(x,y) < v(x,y) \quad \text{for} \ x\in\R, \ y\geq 0.
	\end{equation}
	
	Now define
	\begin{equation*}
	h^*= \sup \{ h\geq h_0 \ : \  \theta \varphi_h(x,y) < v(x,y) \ \text{for} \ x\in\R, \ y\in[h-r, h+r] \}.
	\end{equation*}
	By \eqref{ch11756}, we get that $h^* \geq h_0>r$.
	
	We now take $\tilde{y}<h^*+r$ and define
	\begin{equation*}
	\tilde{h} = \left\{
	\begin{array}{ll}
	\tilde{y}, & \text{if} \ \tilde{y}\leq  h^*, \\
	\dfrac{\tilde{y}+h^*-r}{2}, & \text{if} \ h^* < \tilde{y} < h^*+r. 
	\end{array}
	\right.
	\end{equation*}
	Then, $\tilde{h}<h^*$: if $\tilde{h}=\tilde{y}$ it is trivial, otherwise one observe that $\tilde{y}-r < h^*$.
	Also, $\tilde{y}\in(\tilde{h}-r, \tilde{h}+r )$; in fact, that is obvious if $\tilde{h}=\tilde{y}$, otherwise we have that $\tilde{y}< h^*+r$ and
	\begin{equation*}
	\tilde{h}-r < h^*-r < \tilde{y}  <  \dfrac{\tilde{y}+h^*+r}{2} 
	= \tilde{h}+r.
	\end{equation*}
	Then, since $v_r$ and therefore $\varphi_{\tilde{h}}$ are periodic in $x$, we have that
	\begin{equation} \label{ch11606}
	v(x, \tilde{y}) > \theta \varphi_{\tilde{h}}(x, \tilde{y}) >  \underset{[0,\ell]}{\min} \, \theta \varphi_{\tilde{h}}(x, \tilde{y})>0,
	\end{equation}
	so $v(x,y)>0$ for all $y<h^*+r$, $x\in\R$ and moreover
	\begin{equation}\label{ch11759}
	v(x,y) > \theta \, \underset{[0, \ell]}{\min} \, v_r(x, 0) >0 \quad \text{for} \ x\in\R, \ y\leq h^*.
	\end{equation}
	
	Suppose by absurd that $h^*<+\infty$. Then there exists a sequence $\{h_n\}_{n\in \N}$ and a sequence $\{(x_n, y_n)\}_{n\in\N}$ with $(x_n, y_n)\in \R\times[h_n-r, h_n+r]$, such that $$\underset{n\to+\infty}{\lim} h_n=h^* \quad \text{and} \quad \underset{n\to+\infty}{\lim}   {\theta\varphi_{h_n}(x_n, y_n)-v(x_n, y_n) }=0.$$
	Up to a subsequence, $\{ y_n\}_{n\in\N} \subset [0, h^*+r] $ converges to some $\bar{y}\in[h^*-r, h^*+r]$ while  $\{ x_n\}_{n\in\N}$ either converges to some $\bar{x}\in \R$ or goes to infinity.

	For all $n\in \N$ there exists $x_n'\in[0,\ell)$ and $k\in\Z$ such that
	\begin{equation}\label{ch12040}
	x_n= x_n' + k\ell.
	\end{equation}
	Up to a subsequence,
	\begin{equation}\label{ch11744}
	x_n' \overset{n\to\infty}{\longrightarrow} x'\in[0,\ell].
	\end{equation}
	Define 
	\begin{equation*}
	(u_n(x), v_n(x, y)):=(u(x+x_n), v(x+x_n, y)).
	\end{equation*}
	Then, by Lemma \ref{ch1lemma:converg} we have that $\{(u_n,v_n)\}_{n\in\N}$ converges to some $(\tilde{u}, \tilde{v})$ such that 
	\begin{equation}\label{ch11959}
	\mbox{$(\tilde{u}(x+x'), \tilde{v}(x+x', y)$ is a bounded stationary solution to \eqref{ch1sys:fieldroad}.}
	\end{equation}

	By \eqref{ch11606}, we have
	\begin{equation}\label{ch11955}
	\tilde{v}(x,\tilde{y}) > \underset{x\in[0,\ell]}{\min} \theta \varphi_{\tilde{h}}(x, \tilde{y})>0 \quad \text{for} \ \tilde{y}<h^*+r.
	\end{equation}  
	
	We notice that if $\tilde{v}(0, \bar{y})=0$, since $\tilde{v}\geq 0$ and \eqref{ch11959} holds,
	by the maximum principle  we get $\tilde{v}\equiv 0$ in $\Omega$. But since \eqref{ch11955} holds, this is not possible and instead $\tilde{v}(0, \bar{y})>0$. Hence, $0<\tilde{v}(0, \bar{y})= \theta \varphi_{h^*}(0, \bar{y})$, so
	\begin{equation}\label{ch12050}
	\bar{y}\neq h^*\pm r.
	\end{equation}

	We have that $\theta \varphi_{h_n}$ is a subsolution for $ \mathcal{L}$ in $\R\times (h_n-r, h_n+r)$, since it is a translation of a subsolution. Moreover, thanks to the periodicity of $\varphi_{h_n}$ and the definition of $x_n'$ in \eqref{ch12040}, we have 
	\begin{equation*}
	\varphi_{h_n}(x+x_n, y)=\varphi_{h_n}(x+x_n', y).
	\end{equation*}
	It follows that the sequence $\varphi_{h_n}(x+x_n, y)$ converges to $\varphi_{h^*}(x+{x}', y)$. 
	Then, $\theta \varphi_{h^*}(x+{x}', y)$ is a subsolution for the second equation in \eqref{ch1sys:fieldroad} in $\R\times(h^*-r, h^*+r)$ and  by \eqref{ch12050} it holds $(0, \bar{y})\in \R\times(h^*-r, h^*+r)\subset \Omega$. Hence, we can apply the comparison principle to $\tilde{v}(x,y)$ and $\theta \varphi_{h^*}(x+{x}', y)$: since $\tilde{v}(0,\bar{y})=\theta\varphi_{h^*}({x}', \bar{y})$, then $\tilde{v}(x,y)\equiv\theta \varphi_{h^*}(x+{x}', y) $ in all the points of $\R\times(h^*-r, h^*+r)$. But then by continuity $\tilde{v}(x, h^*-r)=\theta \varphi_{h^*}(x+{x}', h^*-r)=0 $, which is absurd for \eqref{ch11955}. 
	Hence $h^*=+\infty$. From that and \eqref{ch11749} we have statement of the lemma.
\end{proof}

Finally, we are ready to prove existence and uniqueness of a positive bounded stationary solution to \eqref{ch1sys:fieldroad}. The existence of such couple of function is crucial to get the persistence result of Theorem \ref{ch1thm:char}.

\begin{theorem}\label{ch1lemma:pbss}	
	Suppose that $c=0$, $f$ satisfies \eqref{ch1hyp:0}-\eqref{ch1hyp:per} and that
	$\lambda_1( \Omega)<0$. Then, the following holds:
	\begin{enumerate}
		\item There exists a unique positive bounded stationary solution $(u_{\infty}, v_{\infty})$ to system \eqref{ch1sys:fieldroad}.
		\item The functions $u_{\infty}$ and $v_{\infty}$ are periodic in the variable $x$ of period $\ell$.
	\end{enumerate}
\end{theorem}

\begin{proof}
	\emph{Step 1: construction of a subsolution.}
	
	Since $\lambda_1( \Omega)<0$, by Theorem \ref{ch1thm:eq_1_p} it holds that $\lambda_p( \Omega)<0$ and moreover by Proposition \ref{ch1eq:-3}  % ref prop lambda_p( \Omega_R) \to \lambda_p(\Omega)
	there exists $r>1$ such that $\lambda_p( \R\times(0, r))<0$. Let us call $(\phi_r, \psi_r)$ the eigenfunction related to $\lambda_p( \R\times (0,r))$.
	
	We have that
	\begin{equation*}
	\underset{v\to 0^+}{\lim} \dfrac{f(x,v)}{v} = f_v(x,0),
	\end{equation*}
	so there exists $\varepsilon>0$ such that
	\begin{equation*}
	\dfrac{f(x,\varepsilon\psi_r)}{\varepsilon\psi_r} > f_v(x,0) + \lambda_p( \R\times (0,r)).
	\end{equation*}
	Then, 
	\begin{equation}\label{ch11933}
	\left\{
	\begin{array}{ll}
	-D \varepsilon \phi_r'' -c \varepsilon\phi_r' - \nu \varepsilon \psi_r\rvert_{y=0} + \varepsilon \phi_r = \lambda_p ( \R\times(0,r) ) \varepsilon\phi_r < 0,  & x\in\R,  \\
	-d \Delta \varepsilon\psi_r -c \partial_x \varepsilon\psi_r   < f(x, \varepsilon \psi_r),  & (x, y)\in \R\times (0, r), \\
	-d  \varepsilon\partial_y{\psi_r}|_{y=0} + \nu \varepsilon\psi_r|_{y=0} -\varepsilon\phi_r=0, & x\in\R,
	\end{array}
	\right.
	\end{equation}
	so $\varepsilon(\phi_r, \psi_r)$ is a subsolution to system \eqref{ch1sys:fieldroad}.

	%	UTILE?
	%	Secondly, let us observe that by Lemma \ref{ch1lemma:samesign} !!!! FOR  $c=0$ !!!!!! if $\lambda_1( \Omega)<0$ then $\lambda_N(-\mathcal{L}, \Omega)<0$. 
	%	Then, by Proposition \ref{ch1prop:bhr} we also have that $\lambda_p(-\mathcal{L}, \R^2)< \lambda_1(-\mathcal{L}, \R^2)<0$. 
	%	Moreover, by Proposition \ref{ch1prop:eulambdap} we have that $\lambda_p(-\mathcal{L}, \R^2)=\lambda_p(-\mathcal{L}, \R)$ and 
	%
	%	By hypothesis $\lambda_p(-\mathcal{L}, \R^2)< 0$, hence by Lemma (-5)????   $\lambda_p(-\mathcal{L}, \R)< 0$  and
	%
	%	UTILE?
	%	
	%	------
	
	Thanks to Corollary \ref{ch1thm:ineq2}, $\lambda_1( \Omega)<0$ implies $\lambda_1(-\mathcal{L}, \R^2)<0$; then  Proposition \ref{ch1prop:-3} implies that  $\lambda_p(-\mathcal{L}, \R^2)<0$. By \eqref{ch1eq:-5}, also $\lambda_p(-\mathcal{L}, \R)<0$.
	
	Consider the periodic positive eigenfunction $\psi_p(x)$ related to $\lambda_p(-\mathcal{L}, \R)$. 
	With a slight abuse of notation, we can extend $\psi_p(x)$ in all $\R^2$ by considering constant with respect to the variable $y$.
	Repeating the same arguments as before, we can prove that for some $\theta$ the function
	$\theta \psi_p(x)$ is a subsolution for the second equation of system \eqref{ch1sys:fieldroad} in $\R^2$.
	
	Consider $\delta>0$.
	%	\begin{equation*}
	%		\delta:=\min\left\{ \underset{[0,\ell]}{\inf} \, \varepsilon\phi_r(x), \,  \underset{\substack{[0,\ell]\times[0,r-1]}}{\inf} \varepsilon\psi_r(x,y)   \right\}.
	%	\end{equation*}
	We have that $\psi_p(x)$ is limited, therefore there exists $\varepsilon'\in(0, \theta)$ such that 
	\begin{equation}\label{ch11653}
	\underset{[0,\ell]}{\max} \, \varepsilon' \psi_p(x) <\delta < \underset{\substack{[0,\ell]\times[0,r-1]}}{\min} \varepsilon\psi_r(x,y)  .
	\end{equation}
	Then,  let us define the functions
	\begin{align*}
	\underline{u}(x) & :=  \varepsilon\phi_r(x), \\
	\underline{v}(x,y) &:= \max \{  \varepsilon\psi_r(x,y)  , \, \varepsilon' \psi_p(x)\}.
	\end{align*}
	
	By \eqref{ch11653}, for $y\in(0, r-1)$ it holds that $\underline{v}(x,y)=\varepsilon\psi_r(x,y)$. Hence, we get that $(\underline{u}, \underline{v})$ is a subsolution for the first and third equation of \eqref{ch1sys:fieldroad}. Moreover, since $\varepsilon\psi_r(x,y)$ and $\varepsilon' \psi_p(x)$ are both subsolution to the second equation to \eqref{ch1sys:fieldroad}, so  the maximum between them is a generalised subsolution. Thanks to that, we can conclude that $(\underline{u}, \underline{v})$ is a generalised subsolution for the system \eqref{ch1sys:fieldroad}.
	
	Since $\phi_r$ and $\psi_p$ are periodic in $x$ and independent of $y$, we get $$\underset{\R}{\inf} \, \underline{u}(x)>0 \quad \text{and} \quad  \underset{\Omega}{\inf} \, \underline{v}(x,y) >0.$$ 
	
	So, $(\underline{u}, \underline{v})$ is a generalised subsolution for the system \eqref{ch1sys:fieldroad}, with positive infimum, and by the periodicity of $\phi_r$, $\psi_r$ and $\psi_p$, it is periodic in $x$ with period $\ell$. 
	
	\emph{Step 2: construction of a stationary solution.}
	
	Take the generalised subsolution $(\underline{u}, \underline{v})$.	We want to show that the solution $(\tilde{u}(t,x), \tilde{v}(t,x,y))$ having $(\underline{u}(x), \underline{v}(x,y))$ as initial datum is increasing in time and converge to a stationary solution.

	By the fact that $(\underline{u}, \underline{v})$ is a subsolution, at we have $(\underline{u}, \underline{v}) \leq (\tilde{u}, \tilde{v})$ for all $t\geq 0$. Hence, for all $\tau>0$,
	let us consider the solution $(z,w)$ stating at $t=\tau$ from the initial datum $(\underline{u}(x), \underline{v}(x,y))$. Then, at $t=\tau$ we have that $(\tilde{u}(\tau,x), \tilde{v}(\tau,x,y))\geq (z(\tau,x), w(\tau,x,y))$. By the comparison principle \ref{ch1prop:comparison}, we have that for all $t\geq \tau$ it holds that $(\tilde{u}(t,x), \tilde{v}(t,x,y))\geq (z(t,x), w(t,x,y))$. By the arbitrariness of $\tau$, we get that $(\tilde{u}(t,x), \tilde{v}(t,x,y))$ is increasing in time.
	
	Moreover, consider 
	\begin{equation*}
	V:= \max \left\{ M, \sup \underline{v}, \frac{\mu}{\nu} \sup \underline{u} \right\}, \quad U:= \frac{\nu}{\mu} V,
	\end{equation*}
	where $M>0$ is the threshold value defined in \eqref{ch1hyp:M}. 
	One immediatly checks that $(U,V)$ is a supersolution for the system \eqref{ch1sys:fieldroad}.
	Also, we have that $(\underline{u}(x), \underline{v}(x,y)) \leq (U,V)$, so by the comparison principle \ref{ch1prop:comparison} it holds that 
	\begin{equation*}
	(\tilde{u}(t,x), \tilde{v}(t,x,y)) \leq (U,V) \quad \text{for all} \ t>0.
	\end{equation*}
	Hence, $(\tilde{u}(t,x), \tilde{v}(t,x,y))$ is limited. 
	
	Now consider an increasing diverging sequence $\{t_n\}_{n\in N}\subset \R^+$. Then, define
	\begin{equation*}
	u_n(t,x):=\tilde{u}(t+t_n, x), \quad v_n(t,x,y):=\tilde{v}(t+t_n, x, y),
	\end{equation*}
	that is a sequence of functions. By Lemma \ref{ch1lemma:mono},  $(u_n, v_n)$ converge in $\mathcal{C}_{loc}^{1,2,\alpha}$ to a stationary bounded solution to \eqref{ch1sys:fieldroad}, that we call $(u_{\infty}, v_{\infty})$. We point out that $(u_{\infty}, v_{\infty})\not\equiv(0,0)$ since 
	\begin{equation*}
	(u_{\infty}, v_{\infty}) \geq (\underline{u}, \underline{v}) > (0,0).
	\end{equation*}
	Moreover, both functions are periodic of period $\ell$ in the variable $x$ since the initial datum is.
	
	\emph{Step 3: uniqueness.}
	
	Suppose that there exists another positive bounded stationary solution $(q,p)$ to \eqref{ch1sys:fieldroad}. Then, define
	\begin{equation*}
	k^* := \sup \left\{  k>0 \ : \ u_{\infty}(x) > k q(x) \ \forall x\in\R, \ v_{\infty}(x,y)> kp(x,y) \ \forall(x,y)\in\Omega     \right\}.
	\end{equation*}
	Since by Lemma \ref{ch1lemma:pos_inf} the functions $u_{\infty}$ and $v_{\infty}$ have positive infimum and since $p$ and $q$ are bounded, we have that $k^*>0$.
	
	We claim that 
	\begin{equation}\label{ch12337}
	k^* \geq 1.
	\end{equation}	
	
	By the definition of $k^*$, one of the following must hold: there exists
	\begin{equation}\label{ch1caso1}
	\begin{split}
	\mbox{ either a sequence $\{x_n\}_{n\in\N}\subset \R$ such that $u_{\infty}(x_n) - k^* q(x_n) \overset{n\to\infty}{\longrightarrow} 0$,}
	\end{split}	
	\end{equation}
	\begin{equation}\label{ch1caso2}
	\mbox{or a sequence $\{(x_n, y_n)\}_{n\in\N}\subset \Omega$ such that $v_{\infty}(x_n,y_n) - k^* p(x_n, y_n)  \overset{n\to\infty}{\longrightarrow} 0$.}
	\end{equation}
	There exists a sequence $\{ {x}_n' \}_{n\in\N} \subset[0,\ell)$ such that 
	\begin{equation}\label{ch11807}
	x_n-{x}_n' \in \ell \Z \quad \text{for all} \ n\in\N.
	\end{equation}
	Then, up to extraction of a converging subsequence, we have that there exists $ x'\in \R$ such that ${x}_n' \overset{n\to \infty}{\longrightarrow}  x'$.
	One can see that the sequence of couples  
	\begin{equation*}
	(q_n(x), p_n(x,y) ):= (q(x+x_n), p(x+x_n, y) )
	\end{equation*}
	is a stationary solution for \eqref{ch1sys:fieldroad} with reaction function $f(x+ {x}_n', v)$. By Lemma \ref{ch1lemma:converg}, up to a subsequence, $(q_n, p_n)$ converges in $\mathcal{C}_{loc}^{2}$  to some $(q_{\infty}, p_{\infty})$, which is a stationary solution of  \eqref{ch1sys:fieldroad} with reaction function $f(x+  x', v)$. 
	We also notice that, thanks to the periodicity of $u_{\infty}$ and $v_{\infty}$, $(u_{\infty}(x+ x'), v_{\infty}(x+ x', y))$ is also a stationary solution of  \eqref{ch1sys:fieldroad} with reaction function $f(x+  x', v)$. 
	Define the function 
	\begin{equation}
	\begin{split}
	\alpha(x)&:=u_{\infty}(x+ x') - k^*q_{\infty}(x), \\
	\beta(x)&:= v_{\infty}(x+ x', y) - k^*p_{\infty}(x,y),
	\end{split}
	\end{equation}
	and notice that $\alpha (x)\geq 0$, $\beta(x, y)\geq 0$.
	
	Now suppose that \eqref{ch1caso1} holds.
	We have that
	\begin{equation*}
	\alpha(0)=u_{\infty}( x') - k^*q_{\infty}(0)=0.
	\end{equation*}
	Moreover, $\alpha(x)$ is a solution to the equation
	\begin{equation*}
	-D \alpha''-c\alpha' -\nu \beta|_{y=0} +\mu \alpha =0.
	\end{equation*}
	By the maximum principle, we have that since $\alpha(x)$ attains its minimum in the interior of the domain then $\alpha(x)\equiv \min \alpha =0$. Then, one would have $u_{\infty}(x+ x')\equiv k^* q_{\infty}(x)$ and by the comparison principle \ref{ch1prop:comparison} we have $v_{\infty}(x+ x', y)\equiv k^*p_{\infty}(x+ x', y)$. 
	Subtracting the second equation of system \eqref{ch1sys:fieldroad} for $p_{\infty}$ from the one for $v_{\infty} $ we get
	\begin{equation}\label{ch11948}
	0=f(x+ x', v_{\infty}(x+ x',y))-k^* f(x+ x',p_{\infty}(x,y)).
	\end{equation}
	If by the absurd $k^*<1$, by the KPP hypothesis \eqref{ch1hyp:KPP} we have $k^*f(x+  x', p_{\infty}(x,y)) <f(x+  x', k^* p_{\infty}(x,y)) = f(x+ x', v_{\infty}(x+ x', y))$ and the right hand side of \eqref{ch11948} has a sign, that is absurd since the left hand side is 0. We can conclude that if we are in the case of \eqref{ch1caso1}, then \eqref{ch12337} holds.

	Suppose instead that \eqref{ch1caso2} is true. If  $\{ y_n\}_{n\in\N}$  is bounded, we define $y_n \overset{n\to \N}{\longrightarrow} y'\in\R$.
	Then,
	\begin{equation}\label{ch12007}
	\beta(0, y')= v_{\infty}( x', y')- k^* p_{\infty}(0,y' )=0.
	\end{equation}
	If by the absurd $k^*<1$, then by the  Fisher-KPP hypothesis \eqref{ch1hyp:KPP}  we have	
	\begin{equation}\label{ch12001}
	\begin{split}
	-d \Delta \beta(x,y) -c \partial_x \beta(x,y) &= f(x+ x', v_{\infty}(x+ x',y))- k^* f(x+ x', p_{\infty}(x,y)) \\
	&> f(x+ x', v_{\infty}(x+ x',y))- f(x, k^*p_{\infty}(x,y)).
	\end{split}
	\end{equation}
	Since $f$ is locally Lipschitz continuous in the second variable, one infers from \eqref{ch12001} that there exists a bounded function $b(x)$ such that
	\begin{equation}\label{ch12009}
	-d \Delta \beta -c \partial_x \beta + b \beta >0.
	\end{equation}
	Since that, $\beta \geq 0$ and by \eqref{ch12007} $\beta(0, y')=0$, if $y'>0$ we apply the strong maximum principle and we have $\beta \equiv 0$.  
	If $y'=0$, we point out that by the fact that $v_{\infty}$ and $p_{\infty}$ are solution to \eqref{ch1sys:fieldroad} it holds 
	\begin{equation*}
	d \partial_y \beta(x,0) = \nu (v_{\infty}(x+ x',0)  - k^* p_{\infty}(x,0) )  - \nu (u_{\infty}(x+ x')- k^* q_{\infty}(x)) \leq  0
	\end{equation*}
	By that, the inequality in \eqref{ch12009},  $\beta \geq 0$, $\beta(0, y')=0$,  we can apply Hopf's lemma and get again $\beta \equiv 0$.
	Then for both $y'>0$ and $y'=0$, $v_{\infty}(x+ x', y)\equiv k^*p_{\infty}(x+ x', y)$ and \eqref{ch11948} holds, but we have already saw that this is absurd. So, in the case of \eqref{ch1caso2}, if $\{ y_n\}_{n\in\N}$ is bounded, \eqref{ch12337} is true.
	
	At last, if $\{ y_n\}_{n\in\N}$ is unbounded, we define
	\begin{align*}
	V_n(x,y)&:=v_{\infty}(x+x_n, y+y_n), \\
	P_n(x,y)&:=p(x+x_n, y+y_n).
	\end{align*}
	By Lemma \ref{ch1lemma:mono}, up to subsequences, $V_n$ and $P_n$ converge in $\mathcal{C}_{loc}^{2}$  to some functions $V_{\infty}$ and $P_{\infty}$ solving
	\begin{equation*}
	- d \Delta v - c\partial_x v = f(x+ x', v) \quad \text{for} \ (x,y)\in\R^2 .
	\end{equation*}
	Moreover, if we suppose $k^*< 1$, by the Fisher-KPP hypothesis \eqref{ch1hyp:KPP} we have that
	\begin{equation*}
	k^*f(x+ x', P_{\infty})< f(x+ x', k^* P_{\infty})
	\end{equation*}
	and consequently, calling $\gamma:=V_{\infty} - k^* P_{\infty}$, we get
	\begin{equation*}
	- d \Delta\gamma - c\partial_x \gamma > f(x+ x', V_{\infty}) -  f(x+ x', k^*P_{\infty}) .
	\end{equation*}
	Once again using the local Lipschitz boundedness of $f$ in the second variable, for some bounded function $b$ we have that
	\begin{equation}\label{ch12336}
	- d \Delta\gamma - c\partial_x \gamma + b \gamma >0.
	\end{equation}
	Also, we have that
	\begin{equation*}
	\gamma(0,0)=V_{\infty}(0,0) - k^* P_{\infty} (0,0) = \underset{n\to \infty}{\lim} v_{\infty}(x_n, y_n) - k^* p(x_n, y_n)=0.
	\end{equation*}
	Since that, $\gamma \geq 0$ and \eqref{ch12336}, we  can apply the strong maximum principle and we have $\gamma \equiv 0$ in $\R^2$. Then, $V_{\infty}\equiv k^* P_{\infty}$ and 
	\begin{equation}
	0=- d \Delta\gamma - c\partial_x \gamma = f(x+ x', k^*P_{\infty}) -  k^* f(x+ x', P_{\infty}) >0,
	\end{equation}
	which is absurd. Since this was the last case to rule out, we can conclude that \eqref{ch12337} holds.
	
	From \eqref{ch12337}, we  have that 
	\begin{equation}\label{ch12346}
	(u_{\infty}, v_{\infty}) \geq (q,p).
	\end{equation}

	Now, we can repeat all the argument exchanging the role of $(u_{\infty}, v_{\infty})$ and $(q,p)$. We find 
	\begin{equation*}
	h^* := \sup \left\{  h>0 \ : \  q(x) > h u_{\infty}(x)  \ \forall x\in\R, \  p(x,y) >h  v_{\infty}(x,y) \ \forall(x,y)\in\Omega     \right\} \geq 1.
	\end{equation*}
	and
	\begin{equation*}
	(q,p) \geq  (u_{\infty}, v_{\infty}).
	\end{equation*}
	By that and \eqref{ch12346}, we have that $(u_{\infty}, v_{\infty}) \equiv (q,p)$. Hence, the uniqueness is proven.
\end{proof}

Now we are ready to give a result on the persistence of the population.
%What we prove is the more general result, that do not require the hypothesis $c=0$ as in Theorem \ref{ch1thm:char}.
%
%\begin{theorem}
%	Suppose that $f$ satisfies \eqref{ch1hyp:0}-\eqref{ch1hyp:per} and that
%	$\lambda_1( \Omega)<0$. Then, persistence occurs and the positive stationary solution is unique and periodic in $x$.
%\end{theorem}

\begin{proof}[Proof of Theorem \ref{ch1thm:char}, part 1]

	Since $\lambda_1( \Omega)<0$, by Proposition \eqref{ch1prop:romain}, we have that there exists $R>0$ such that $\lambda_1( \Omega_R)<0$. Let us consider $(u_R, v_R)$ the eigenfunctions related to $\lambda_1( \Omega_R)<0$; then, with the argument already used in  the proof of Lemma \ref{ch1lemma:pos_inf} (precisely, in \eqref{ch12220} and \eqref{ch12221}), there exists a value $\varepsilon>0$ such that $(\varepsilon u_R, \varepsilon v_R)$ is a subsolution to \eqref{ch1sys:fieldroad} in $\Omega_R$. 
	Observe also that $u_R(x)=0$ for $x\in\partial I_R$ and $v_R(x,y)=0$ for $(x,y)\in(\partial \Omega_R)\cap\Omega$. Then, we can extend $\varepsilon u_R$ and $ \varepsilon v_R$ outside respectively $I_R$ and $\Omega_R$, obtaining the generalised subsolution $(\varepsilon u_R, \varepsilon v_R)$.
	
	Let us consider the solution $(u,v)$ issued from $(u_0, v_0)$. 
	Then, by the strong parabolic principle we have that
	\begin{equation}\label{ch12225}
	u(1, x)>0\quad \text{and} \quad  v(1,x,y)>0.
	\end{equation}
	Recall that $(u_{\infty}, v_{\infty})$ is the unique stationary solution of \eqref{ch1sys:fieldroad}, and that by Lemma \eqref{ch1lemma:pos_inf} we have
	\begin{equation}\label{ch12300}
	u_{\infty}>0\quad \text{and} \quad  v_{\infty}>0.
	\end{equation}
	By that and \eqref{ch12225},  we have that
	\begin{equation*}
	\delta:=\min\{\underset{x\in I_R}{\min} \, u(1,x), \underset{x\in I_R}{\min} \, u_{\infty}(x), 	\underset{(x,y)\in \Omega_R}{\min}  v(1,x,y), \underset{(x,y)\in \Omega_R}{\min} v_{\infty}(x,y) \} >0.
	\end{equation*}
	Without loss of generality, we can suppose 
	\begin{equation}\label{ch12301}
	\varepsilon <\delta
	\end{equation}
	and thus by \eqref{ch12225}, \eqref{ch12300}, and \eqref{ch12301}, we have
	\begin{equation}\label{ch12306}
	\begin{split}
	u_{\infty}(x) &> \varepsilon u_R(x) \quad  \text{for all}  \ x\in \R, \\
	v_{\infty}(x,y) &> \varepsilon v_R(x,y) \quad   \text{for all} \ (x,y)\in \Omega.
	\end{split}
	\end{equation}

	Now, consider the solution $(\underline{u}, \underline{v})$ issued from $(\varepsilon u_R, \varepsilon v_R)$.
	We point out that, by the comparison principle, for all $t>0$ we have
	\begin{equation}\label{ch10017}
	(\underline{u}(t,x), \underline{v}(t,x,y) \leq ({u}(t+1,x), {v}(t+1,x,y)).
	\end{equation}
	By the standard argument already used in the proof of Theorem \ref{ch1lemma:pbss}, we have that $(\underline{u}, \underline{v})$ is increasing in time and by Lemma \ref{ch1lemma:mono} it converges in $\mathcal{C}_{loc}^2$ to a stationary function $(\underline{u_{\infty}}, \underline{v_{\infty}})$ as $t$ tends to infinity. Since $(\underline{u}, \underline{v})$ is increasing in time and $(\varepsilon u_R, \varepsilon v_R)\not \equiv (0,0)$, by the strong maximum principle we have $(\underline{u_{\infty}}, \underline{v_{\infty}}) > (0,0)$. By \eqref{ch12306}, we also have 
	\begin{equation*}
	(\underline{u_{\infty}}, \underline{v_{\infty}}) \leq ({u_{\infty}}, {v_{\infty}})
	\end{equation*}
	Then, by the uniqueness of the bounded positive stationary solution proved in Theorem \ref{ch1lemma:pbss}, we have $(\underline{u_{\infty}}, \underline{v_{\infty}}) \equiv ({u_{\infty}}, {v_{\infty}})$.  
	
	Next, take  
	\begin{equation}\label{ch10008}
	V:= \max \left\{ M, \sup v_0, \frac{\mu}{\nu} \sup u_0, \sup v_{\infty}, \frac{\mu}{\nu} \sup u_{\infty}  \right\}, \quad U:= \frac{\nu}{\mu} V,
	\end{equation}
	where $M>0$ is the threshold value defined in \eqref{ch1hyp:M}.
	Making use of the hypothesis \eqref{ch1hyp:M} on $f$, one easily check that $(U, V)$ is a supersolution for \eqref{ch1sys:fieldroad}. Let us call $(\overline{u}, \overline{v})$ the solution to \eqref{ch1sys:fieldroad} issued from $(U,V)$.
	By definition, $(U,V)\geq (u_0, v_0)$, hence by the comparison principle for all $t>0$ we have
	\begin{equation}\label{ch10012}
	( u(t,x), v(t, x,y)  ) \leq (\overline{u}(t,x), \overline{v}(t, x,y) ).
	\end{equation}
	Repeating the argument used in the proof of Theorem \ref{ch1lemma:pbss}, we observe that $(\overline{u}, \overline{v})$ is decreasing in time and by Lemma \ref{ch1lemma:mono} it converges in $\mathcal{C}_{loc}^2$ to a stationary function $(\overline{u_{\infty}}, \overline{v_{\infty}})$ as $t$ tends to infinity.
	We have $(\overline{u_{\infty}}, \overline{v_{\infty}}) \leq (U,V)$, so the stationary solution is bounded. 
	Moreover, since by the definition of $(U,V)$ in \eqref{ch10008} we have $ ( {u_{\infty}},  {v_{\infty}}) \leq (U, V) $, by the comparison principle \ref{ch1prop:comparison} we get
	\begin{equation*}
	( {u_{\infty}},  {v_{\infty}}) \leq (\overline{u_{\infty}}, \overline{v_{\infty}}).
	\end{equation*}
	Since $(\overline{u_{\infty}}, \overline{v_{\infty}})$ is a bounded positive stationary solution of \eqref{ch1sys:fieldroad}, by Theorem \ref{ch1lemma:pbss} we have that $( {u_{\infty}},  {v_{\infty}}) \equiv (\overline{u_{\infty}}, \overline{v_{\infty}})$.
	
	By the comparison principle \ref{ch1prop:comparison} and by \eqref{ch10017} and \eqref{ch10012}, for all $t>1$ we have
	\begin{equation*}
	\begin{split}
	\underline{u}(t-1, x) \leq u(t,x) \leq \overline{u}(t,x) \quad \text{for all} \ x\in\R, \\
	\underline{v}(t-1, x,y) \leq v(t,x,y) \leq \overline{v}(t,x,y) \quad \text{for all} \ (x,y)\in\Omega.
	\end{split}
	\end{equation*}
	Since both $(\underline{u}, \underline{v})$ and $(\overline{u}, \overline{v})$ converge to $( {u_{\infty}},  {v_{\infty}})$ locally as $t$ tends to infinity, by the sandwich theorem we have that $(u, v)$ also does. This is precisely the statement that we wanted to prove. 
\end{proof}

\subsection{Extinction}

The first step to prove extinction is to show that there is no positive bounded stationary solution to system \eqref{ch1sys:fieldroad}, that is, the only bounded stationary solution is $(0,0)$.

\begin{lemma}\label{ch1lemma:ext}
	Suppose $c=0$ and $f$ satisfy \eqref{ch1hyp:0}-\eqref{ch1hyp:per}.
	If $\lambda_1( \Omega)\geq 0$, then there is no positive bounded stationary solution to system \eqref{ch1sys:fieldroad}.
\end{lemma}

\begin{proof}
	\emph{Step 1: construction of a supersolution.}
	Observe that in this case, since $c=0$, by Theorem \ref{ch1thm:eq_1_p} it holds $\lambda_p(\Omega)=\lambda_1(\Omega)\geq 0$.
	We take the couple of eigenfunctions $(u_p, v_p)$ related to $\lambda_p(\Omega)$ as prescribed by Proposition \ref{ch1prop:-3}; recall that $(u_p, v_p)$ are periodic in $x$.
	Suppose $(q,p)$ is a positive bounded stationary solution to \eqref{ch1sys:fieldroad}. Then, there exists $\eta>0$ such that
	\begin{equation}\label{ch12347}
	q(0) > \eta u_p(0).
	\end{equation}
	We now choose a smooth function $\chi :  \R_{\geq 0} \to \R_{\geq 0}$  such that  $\chi(y)=0$ for $y\in[0,\ell]$, $\chi(y)=1$ for $y\in[ 2\ell, +\infty)$. %, and there exists a constant $k>0$ such that 
	%	\begin{equation}\label{ch11950}
	%	\chi ''(y)\leq k.
	%	\end{equation}
	
	By \eqref{ch1eq:-5} and Theorem \ref{ch1thm:1.7inbr}, we have $\lambda_p(-\mathcal{L}, \R)=\lambda_p(-\mathcal{L}, \R^2)=\lambda_1(-\mathcal{L}, \R^2)$. By that, Theorem \ref{ch1thm:ineq2} and the fact that $\lambda_1(\Omega)\geq 0$, we get $\lambda_p(-\mathcal{L}, \R)\geq 0$.
	We call $\psi_p$ the eigenfunction related to $\lambda_p(-\mathcal{L}, \R)$ and, with a slight abuse of notation, we extend it to $\R^2$ by considering it constant with respect to the variable $y$.
	Take $\varepsilon>0$ to be fixed after, and define
	\begin{equation*}
	(\overline{u}(x),  \overline{v}(x,y)):= (\eta u_p(x),  \eta v_p(x,y) + \varepsilon \chi(y) \psi_p(x)).
	\end{equation*}
	Then, it holds that
	\begin{equation}\label{ch12011}
	\begin{split}
	- d \Delta \overline{v}  
	&= -d \left(\Delta \eta v_p +\varepsilon \chi''\psi_p + \varepsilon \chi \psi_p''   \right), \\
	&= \left( f_v(x,0)+\lambda_p( \Omega) \right) \eta v_p + ( f_v(x,0)+ \lambda_p(-\mathcal{L}, \R) ) \varepsilon \chi \psi_p - d \varepsilon \chi''\psi_p, \\
	& = f_v(x,0) \overline{v} + \lambda_p( \Omega) \eta v_p + \lambda_p(-\mathcal{L}, \R) \varepsilon \chi \psi_p - d \varepsilon \chi''\psi_p.
	%	&> f(x, \overline{v}) +  \lambda_p( \Omega)\eta  v_p + \lambda_p(-\mathcal{L}, \R) \varepsilon \chi \psi_p , \\
	%	& \geq f(x, \overline{v}).
	\end{split}
	\end{equation}
	Using the KPP hypothesis \eqref{ch1hyp:KPP} and the boundedness of $\chi''$, for $\varepsilon$ small enough we have
	\begin{equation*}
	f_v(x,0) \overline{v} - d \varepsilon \chi''\psi_p > f(x, \overline{v}).
	\end{equation*}
	By that, \eqref{ch12011} and the non negativity of $\lambda_p(\Omega)$ and $\lambda_p(-\mathcal{L}, \R)$, we have 
	\begin{equation*}
	- d \Delta \overline{v} > f(x, \overline{v}).
	\end{equation*}
	This means that $\overline{v}$ is a supersolution for the second equation of \eqref{ch1sys:fieldroad}. 
	
	Since by definition for $y\leq \ell$ we have $\chi(y)=0$, it holds that
	\begin{equation}\label{ch12010}
	(\overline{u}(x),  \overline{v}(x,y))\equiv (u_p(x),   v_p(x,y)) \quad \text{for all} \ (x, y)\in \R\times (0, \ell). 
	\end{equation} 
	By the fact that $\lambda_p(\Omega  ) \geq 0$, it is easy to check that $(u_p(x),   v_p(x,y))$ is a supersolution for the first and third equation in \eqref{ch1sys:fieldroad}. By \eqref{ch12010}, the same holds for  $(\overline{u}(x),  \overline{v}(x,y))$.
	This, together with \eqref{ch12011}, gives that $(\overline{u}(x),  \overline{v}(x,y))$ is a supersolution to \eqref{ch1sys:fieldroad}.

	\emph{Step 2: construction of a bounded supersolution}
	Now we distinguish two cases.
	If $ v_p$ is bounded, then we take
	\begin{equation}\label{ch1case1super}
	(\tilde{u}, \tilde{v}):= (\bar{u}, \bar{v})
	\end{equation}
	Otherwise, we proceed as follows.
	Since in this other case $v_p$ is unbounded, and since it is periodic in $x$, this means there exists a sequence $\{(x_n, y_n)\}_{n\in\N}$ such that 
	\begin{equation}\label{ch11721b}
	v_p(x_n, y_n) \to \infty, \ y_n\to \infty \quad \text{as} \ n\to\infty.
	\end{equation}
	Now, consider 
	\begin{equation}\label{ch11418}
	V:= \max \left\{ \underset{[0,\ell]\times [0, 3\ell]}{\max} v_p +1, \ \underset{[0,\ell]}{\max} \, \frac{\nu}{\mu} u_p +1 ,  \ M       \right\},
	\end{equation}
	where $M$ is the quantity defined in \eqref{ch1hyp:M}.
	Take the set $S:=(-\ell, \ell)\times(-\ell, \ell)$ and the constant $C$ of the Harnack inequality (see Theorem 5 in Chapter 6.4 of \cite{evans}) on the set $S$ for the operator $L(\psi)=\mathcal{L}(\psi)+\lambda_1(\Omega)\psi$. Then, by \eqref{ch11721b}, for some $N\in\N$ we have  
	\begin{equation*}
	V \leq \frac{1}{C} v_p(x_N, y_N).
	\end{equation*}
	Then by using that and Harnack inequality on $v_p(x+x_N,y+y_N)$ in the set $S$, we get 
	\begin{equation*}
	V \leq \frac{1}{C} \, \underset{S}{\sup} \, v_p(x, y)
	\leq \underset{S}{\inf} \, v_p(x,y),
	\end{equation*}
	Then, using the periodicity of $v_p$, we get
	\begin{equation}\label{ch1comp1}
	V \leq v_p(x, y_N)  \quad \text{for all} \ x\in\R.
	\end{equation}
	Now, define
	\begin{equation}\label{ch1case2superv}
	\tilde{v}(x,y):= \left\{
	\begin{array}{ll}
	\min \{  V, \bar{v}(x,y) \} & \text{if} \ y \leq y_N, \\
	V & \text{if} \ y > y_N.
	\end{array}
	\right.
	\end{equation}
	Also, we define
	\begin{equation*}
	U := \frac{\nu}{\mu} V
	\end{equation*}
	and
	\begin{equation}\label{ch1case2superu}
	\tilde{u}:= \min\{ U, u_p  \}.
	\end{equation}
	By the definition of $V$ in \eqref{ch11418}, one readily checks that $(U,V)$ is a supersolution for system \eqref{ch1sys:fieldroad} and that 
	\begin{equation}\label{ch1comp2}
	\tilde{u} = u_p \quad \text{and} \quad \tilde{v}(x,0)=v_p(x,0).
	\end{equation}
	We point out that by the definition of $(\tilde{u}, \tilde{v})$, \eqref{ch1comp1} and \eqref{ch1comp2}, for any $(\underline{u}, \underline{v})$ subsolution to system \eqref{ch1sys:fieldroad}, we will be able to apply
	the generalised comparison principle, Proposition 3.3 appeared in \cite{brr}. 
	Moreover, $(\tilde{u}, \tilde{v})$ is bounded from above by $(U,V)$.

	By the fact that $(u_p,  v_p)$ is a couple of generalised periodic eigenfunctions to \eqref{ch1sys:upvp}, by the strong maximum principle we have that %$u_p(x)>0$ for $x\in[0,\ell]$ and $ v_p(x,y)>0$ for $(x, y)\in[0,\ell]\times [0, 2\ell]$. Then, we have that
	\begin{equation}\label{ch12341}
	\begin{split}
	\tilde{u}(x) &\geq \underset{[0,\ell]}{\min} \, \eta u_p(x') >0 \quad \text{for} \ x\in\R, \\
	\tilde{v}(x,y) &\geq \underset{ [0,\ell]\times[0,2\ell]}{\min} \eta  v_p(x',y') >0 \quad \text{for} \ (x,y)\in\R\times[0, 2\ell], \\
	\tilde{v}(x,y) &\geq \min\{\underset{[0,\ell]}{\min} \, \varepsilon\psi_p(x'), V\} >0 \quad \text{for} \ (x,y)\in\R\times(2\ell, +\infty).
	\end{split}
	\end{equation}
	
	\emph{Step 3: comparison with the stationary solution.}
	Next, define
	\begin{equation*}
	k^*:= \inf \{ k\geq 0 \ : \ k( \tilde{u}(x),   \tilde{v}(x,y)) > (q,p) \ \text{for all} \ (x,y)\in\Omega       \}.
	\end{equation*}
	Since by \eqref{ch12341} we have that $ \tilde{u}(x)$ and $ \tilde{v}(x,y)$ are bounded away from $0$, and since $(q,p)$ is bounded by hypothesis, we get that $k^*<+\infty$.
	By \eqref{ch12347}, we have that
	\begin{equation}\label{ch12234}
	k^*>1.
	\end{equation}
	Then, either
	\begin{equation}\label{ch1case1}
	\mbox{there exists a sequence $\{x_n\}_{n\in\N}\subset \R$ such that $k^*  \tilde{u}(x_n) - q(x_n) \overset{n\to\infty}{\longrightarrow} 0$,}
	\end{equation}
	or
	\begin{equation}\label{ch1case2}
	\mbox{
		there exists a sequence $\{(x_n, y_n)\}_{n\in\N}\subset \Omega$ such that $k^*  \tilde{v}(x_n,y_n) -  p(x_n, y_n)  \overset{n\to\infty}{\longrightarrow} 0$.}
	\end{equation}

	As usual, for all $n\in\N$ we take $x_n'\in[0,\ell)$ such that $x_n-x_n'\in\ell \Z$. Up to a subsequence, $\{ x_n'\}_{n\in\N}$ is convergent and we call
	\begin{equation*}
	x'= \underset{n\to\infty}{\lim} x_n' \in[0,\ell].
	\end{equation*}

	\emph{Step 4: $\{y_n\}_{n\in\N}$ is bounded}.
	If $\{y_n\}_{n\in\N}$ is bounded, consider a converging subsequence and call $y'= \underset{n\to \infty}{\lim} y_n$.
	
	We define 
	\begin{equation*}
	(q_n(x), p_n(x,y)):=(q(x+x_n), p(x+x_n,y)). 
	\end{equation*}
	By Lemma \ref{ch1lemma:converg}, $(q_n, p_n)$ converges in $\mathcal{C}_{loc}^2$ to some $(q_{\infty}, p_{\infty})$ such that $(q_{\infty}(x-x'), p_{\infty}(x-x', y))$ solves \eqref{ch1sys:fieldroad}.
	Define the functions 
	\begin{align*}
	\alpha(x) &:= k^*  \tilde{u}(x)-q_{\infty}(x-x'), \\
	\beta(x,y)&: = \tilde{v}(x,y)- p_{\infty}(x-x', y).
	\end{align*}
	
	If we are in the case of \eqref{ch1case1}, then by the periodicity of $\tilde{u}$ we get
	\begin{equation*}
	\alpha(x')= k^*  \tilde{u}(x')-q_{\infty}(0)= \underset{n\to\infty}{\lim} (  k^*  \tilde{u}(x_n)- q(x_n)   )=0.
	\end{equation*}
	Moreover, by the definition of $k^*$, we have that $\alpha\geq 0$. Also, $\alpha$ satisfies
	\begin{equation*}
	-D \alpha ''  -\nu \beta|_{y=0}+ \nu \alpha \geq 0. 
	\end{equation*}
	Then, the strong maximum principle yields that, since $\alpha$ attains its minimum at $x=x'$, then $\alpha\equiv0$. Then, by the comparison principle 3.3 in \cite{brr} we have that $\beta\equiv 0$, hence 
	\begin{equation}\label{ch12226}
	0= -d \Delta \beta \geq k^*f(x,  \tilde{v}) - f (x,p_{\infty}(x-x',y)).
	\end{equation}
	By \eqref{ch12234}, we have that $k^*  \tilde{v} >  \tilde{v}$. Hence, by the Fischer-KPP hypothesis \eqref{ch1hyp:KPP}, we have that
	\begin{equation}\label{ch12249}
	\frac{f(x, k^*  \tilde{v})}{k^*  \tilde{v}} < \frac{f(x,  \tilde{v})}{ \tilde{v}}.
	\end{equation}
	Hence, again by the fact that $\beta\equiv 0$, we  have $p_{\infty}(x-x',y)\equiv  k^* \tilde{v}$; by that and by \eqref{ch12249}, it holds
	\begin{equation}\label{ch12257}
	k^* f(x,  \tilde{v})- f(x,p_{\infty}(x-x',y))=k^* f(x,  \tilde{v})- f(x, k^* \tilde{v})>0.
	\end{equation}
	But this is in contradiction with \eqref{ch12226}, hence this case cannot be possible. 
	
	If instead \eqref{ch1case2} holds, we get that 
	\begin{equation}\label{ch12256}
	\beta(x', y')= k^*  \tilde{v}(x', y')- p_{\infty}(0, y')= \underset{n\to \infty}{\lim} k^*  \tilde{v}(x_n, y_n)- p(x_n, y_n)=0.
	\end{equation}
	By the definition of $k^*$ we also have that $\beta \geq 0$. Moreover, we get that
	\begin{equation*}
	-d \Delta \beta \geq f(x,k^*  \tilde{v}) - f (x,p_{\infty}(x-x',y))
	\end{equation*}
	using the fact that $ \tilde{v}(x,y)$ is a supersolution, $p_{\infty}(x-x',y)$ is a solution, and \eqref{ch12249}.
	Since $f$ is Lipschitz in the second variable, uniformly with respect to the first one, there exists some function $b$ such that
	\begin{equation*}
	-d \Delta \beta  - b \beta \geq 0.
	\end{equation*}
	If $y'>0$, using the strong maximum principle and owing \eqref{ch12256}, we have that $\beta \equiv 0$. 
	If instead $y'=0$, recall that it also holds
	\begin{equation*}
	-d \partial_y \beta|_{y=0} \geq \mu \alpha -\nu \beta.
	\end{equation*}
	Hence, in $(x,y)=(x', y')$, we get that $\partial_y \beta(x',y') \leq 0$. By Hopf's lemma, we get again that $\beta\equiv 0$.
	
	But $\beta\equiv 0$ leads again to  \eqref{ch12226} and \eqref{ch12257}, giving an absurd, hence also this case is not possible.

	\emph{Step 5: $\{y_n\}_{n\in\N}$ is unbounded}.
	We are left with the case of $\{y_n\}_{n\in\N}$ unbounded.
	Up to a subsequence, we can suppose that $\{y_n\}_{n\in\N}$ is increasing.
	We define
	\begin{equation*}
	P_n(x,y):= p(x+x_n, y+y_n).
	\end{equation*}
	By Lemma \ref{ch1lemma:converg} we have that, up to a subsequence, $\{P_n\}_{n\in\N}$ converges in $\mathcal{C}_{loc}^{2,\alpha}(\R^2)$ to some function $P_{\infty}$  such that $P_{\infty}(x-x',y)$ is a solution  to the second equation in \eqref{ch1sys:fieldroad} in $\R^2$.

	% Qui devo usare il fatto che \bar{v} sia limitato
	Now we have two cases depending on how $(\tilde{u}, \tilde{v})$ was constructed. If $v_p$ is bounded, we have defined the supersolution as in \eqref{ch1case1super}. Then,
	by defining 
	\begin{equation*}
	v_n(x,y):=v_p(x+x_n, y+y_n)
	\end{equation*}
	and applying Lemma \ref{ch1lemma:converg}, we have that $v_n$ converges locally uniformly to  a bounded function $v_{p, \infty}$ such that $v_{p, \infty}(x-x',y)$ satisfies
	\begin{equation}\label{ch11630}
	-d\Delta v_{p, \infty}(x-x',y) = (f_v(x,0)+\lambda_1(\Omega))v_{p, \infty}(x-x',y).
	\end{equation}
	In this case, we define 
	\begin{equation*}
	v_{\infty}(x,y):=\eta v_{p, \infty}(x,y) + \varepsilon \psi_p(x+x').
	\end{equation*}
	We point out that $v_{\infty}(x-x',y)$ is a  periodic supersolution of the second equation in \eqref{ch1sys:fieldroad} by \eqref{ch11630} and \eqref{ch12011}.
	
	If instead $v_p$ is unbounded, by \eqref{ch1case2superu} for $y>y_N$ we have $\tilde{v}=V$. In this case, we choose
	\begin{equation*}
	v_{\infty}:=V.
	\end{equation*}
	By the definition of $V$ in \eqref{ch11418}, we have that $v_{\infty}$ is also a supersolution to \eqref{ch1sys:fieldroad}.
	
	%	We have that $ \tilde{v}(x,y)= \varepsilon \psi_p(x)$ for all $y\geq 3\ell$, and that $ \tilde{v}(x,y)$ is periodic in $x$.
	We call $\gamma(x,y):=k^* {v}_{\infty}(x-x',y) - P_{\infty}(x-x',y)$.
	Hence, $\gamma(x,y)\geq 0$ and
	\begin{equation}\label{ch12330}
	\gamma(x', 0)=k^* {v}_{\infty}(0, 0) - P_{\infty}(0,0)= \underset{n\to\infty}{\lim} k^*  \tilde{v}(x_n,y_n) - p(x_n, y_n)=0. 
	\end{equation}
	Notice than that, since \eqref{ch12234} holds, from the Fisher-KPP hypothesis on $f$ \eqref{ch1hyp:KPP}, we get
	\begin{equation*}
	\frac{f(x,k^* {v}_{\infty} )}{k^* {v}_{\infty}} < \frac{f(x, {v}_{\infty} )}{ {v}_{\infty}}.		
	\end{equation*}
	Using that, the fact that $k^* {v}_{\infty}(x-x',y)$ is a supersolution, and the fact that $P_{\infty}(x-x',y)$ is a solution, we obtain
	\begin{equation}\label{ch12333}
	-d \Delta \gamma  > f(x,  k^*{v}_{\infty}(x-x',y)) - f (x,P_{\infty}(x-x',y)).
	\end{equation}
	Since $f$ is Lipschitz in the second variable, uniformly with respect to the first one, there exists some function $b$ such that
	\begin{equation*}
	-d \Delta \gamma  - b \gamma \geq 0.
	\end{equation*}
	Using the strong maximum principle for the case of positive functions, since \eqref{ch12330} holds, we have that $\gamma\equiv 0$. As a consequence, from  \eqref{ch12333} we have
	\begin{equation*}
	f(x,k^*  {v}_{\infty}) - f (x,P_{\infty}) <0.
	\end{equation*} 
	but it also holds that $k^*  {v}_{\infty} \equiv P_{\infty}$, hence we have an absurd.

	Having ruled out all the possible cases, we can conclude that there exists no bounded positive stationary solution $(q,p)$ to \eqref{ch1sys:fieldroad}.
\end{proof}

At last, we are ready to prove the first part of Theorem \ref{ch1thm:char}.

\begin{proof}[Proof of Theorem \ref{ch1thm:char}, part 1]
	Define $$V:=\max \left\{ M, \sup v_0, \frac{\mu}{\nu} \sup u_0  \right\} \quad \text{and} \quad U:= \frac{\nu}{\mu} V.$$
	It is easy to check that $(U, V)$ is a supersolution for \eqref{ch1sys:fieldroad}. Then take $(\overline{U}, \overline{V})$ to be the solution to \eqref{ch1sys:fieldroad} with initial datum $(U,V)$. 
	Notice that by the comparison principle 
	\begin{equation}\label{ch10013}
	(0,0)\leq (u(t,x), v(t,x,y)) \leq (\overline{U}(t,x), \overline{V}(t,x,y)) \quad \text{for all} \ t>0, \ (x,y)\in\Omega.
	\end{equation}
	
	Since $(U, V)$ is a supersolution, we have that
	\begin{equation}\label{ch12317}
	(\overline{U}, \overline{V}) \leq (U,V) \quad \text{for all} \ t\geq 0.
	\end{equation} 
	Consider $\tau>0$ and call $(\tilde{U}, \tilde{V})$ the solution staring with initial datum $(U, V)$ at $t=\tau$. By \eqref{ch12317} we have that $(\overline{U}(\tau,x), \overline{V}(\tau,x, y)) \leq (U,V)$, hence by the comparison principle \eqref{ch1prop:comparison} we have  $(\overline{U}, \overline{V}) \leq (\tilde{U}, \tilde{V})$. By the arbitrariness of $\tau$, we get that $(\overline{U}, \overline{V})$ is decreasing in $t$.
	
	By Lemma \ref{ch1lemma:mono},  $(\overline{U}, \overline{V})$ converges locally uniformly to a stationary solution $(q,p)$. But by Lemma \ref{ch1lemma:ext}, the only stationary solution is $(0,0)$.
	By that and \eqref{ch10013}, we have that $(u(t,x), v(t,x,y))$ converges locally uniformly to $(0,0)$ as $t$ goes to infinity.
	
	Moreover, since $(U,V)$ is constant in $x$, and \eqref{ch1sys:fieldroad} is periodic in x, $(\overline{U}, \overline{V})$ is periodic in $x$.
	Hence, the convergence is uniform in $x$.

	Now suppose by the absurd that the convergence is not uniform in $y$; hence there exists some $\varepsilon>0$ such that for infinitely many $t_n\geq 0$, with $\{t_n\}_{n\in\N}$ an increasing sequence, and $( x_n, y_n)\in\Omega$, it holds  
	\begin{equation}\label{ch12327}
	\overline{V}(t_n,  x_n, y_n) >\varepsilon.
	\end{equation}
	Since $\overline{V}$ is periodic in $x$, without loss of generality we can suppose $x_n\in[0,\ell]$ and that up to a subsequence $\{x_n\}_{n\in\N}$ converges to some $x'\in[0,\ell]$. If $\{y_n\}_{n\in\N}$ were bounded, by \eqref{ch12327} the local convergence to $0$ would be contradicted; hence $y_n$ is unbounded.
	
	Then, define the sequence of functions
	\begin{equation*}
	V_n(t,x,y)=\overline{V}(t, x+x_n, y+y_n).
	\end{equation*}
	By \eqref{ch12327}, we have that
	\begin{equation}\label{ch10109}
	V_n(t_n,0,0)>\varepsilon \quad \text{for all} \ n\in\N.
	\end{equation}
	
	Also, since $V_n$ is bounded, by arguments similar to the ones used in Lemma \ref{ch1lemma:converg} and Lemma \ref{ch1lemma:mono}
	one can prove that, up to a subsequence, 
	$\{V_n\}_{n\in\N}$ converges in $\mathcal{C}_{loc}^2(\R^2)$ to a function $\tilde{V}$  that   solves
	\begin{equation}\label{ch10108}
	\partial_t \tilde{V} - d\Delta \tilde{V}  = f(x+x', \tilde{V}).
	\end{equation}
	Also by \eqref{ch10109}, we have that
	\begin{equation}\label{ch10110}
	\tilde{V}(t_n,0,0 )>\varepsilon \quad \text{for all} \ n\in\N.
	\end{equation}
	Recall that by the fact that $\lambda_1(\Omega)\geq 0$, Corollary \ref{ch1thm:ineq2} and Theorem \ref{ch1thm:1.7inbr}, $\lambda_p(-\mathcal{L}, \R^2)\geq 0$. Then by Theorem \ref{ch1thm:2.6inbhroques} we have that every solution to \eqref{ch10108} converges uniformly to $0$. But this is in contradiction with \eqref{ch10110}, hence we have an absurd and we must refuse the existence such positive $\varepsilon$. So, the convergence of $\overline{V}$ to $0$ is uniform in space.
	As a consequence, the convergence of $(u(t,x), v(t,x,y))$ to $(0,0)$ is uniform in space.
\end{proof}


\begin{thebibliography}{50}
	
	\bibitem{weinberger2}
	D.~G. Aronson and H.~F. Weinberger.
	\newblock Multidimensional nonlinear diffusion arising in population genetics.
	\newblock {\em Advances in Mathematics}, 30(1):33--76, 1978.
	
	\bibitem{berestycki2014speed}
	H.~Berestycki, A.-C. Coulon, J.-M. Roquejoffre, and L.~Rossi.
	\newblock Speed-up of reaction-diffusion fronts by a line of fast diffusion.
	\newblock {\em S{\'e}minaire Laurent Schwartz—EDP et applications}, pages
	1--25, 2014.
	
	\bibitem{berestycki2015effect}
	H.~Berestycki, A.-C. Coulon, J.-M. Roquejoffre, and L.~Rossi.
	\newblock The effect of a line with nonlocal diffusion on fisher-kpp
	propagation.
	\newblock {\em Mathematical Models and Methods in Applied Sciences},
	25(13):2519--2562, 2015.
	
	\bibitem{berestycki2009can}
	H.~Berestycki, O.~Diekmann, C.~J. Nagelkerke, and P.~A. Zegeling.
	\newblock Can a species keep pace with a shifting climate?
	\newblock {\em Bulletin of mathematical biology}, 71(2):399, 2009.
	
	\bibitem{romain}
	H.~Berestycki, R.~Ducasse, and L.~Rossi.
	\newblock Generalized principal eigenvalues for heterogeneous road-field
	systems.
	\newblock {\em arXiv preprint arXiv:1810.13180}, 2018.
	
	\bibitem{econiches}
	H.~Berestycki, R.~Ducasse, and L.~Rossi.
	\newblock Influence of a road on a population in an ecological niche facing
	climate change.
	\newblock {\em arXiv preprint arXiv:1903.02221}, 2019.
	
	\bibitem{bhroques}
	H.~Berestycki, F.~Hamel, and L.~Roques.
	\newblock Analysis of the periodically fragmented environment model: I--species
	persistence.
	\newblock {\em Journal of Mathematical Biology}, 51(1):75--113, 2005.
	
	\bibitem{BRR2}
	H.~Berestycki, J.-M. Roquejoffre, and L.~Rossi.
	\newblock Fisher--kpp propagation in the presence of a line: further effects.
	\newblock {\em Nonlinearity}, 26(9):2623, 2013.
	
	\bibitem{brr}
	H.~Berestycki, J.-M. Roquejoffre, and L.~Rossi.
	\newblock The influence of a line with fast diffusion on fisher-kpp
	propagation.
	\newblock {\em Journal of mathematical biology}, 66(4-5):743--766, 2013.
	
	\bibitem{br2}
	H.~Berestycki and L.~Rossi.
	\newblock Reaction-diffusion equations for population dynamics with forced
	speed i-the case of the whole space.
	\newblock {\em Discrete \& Continuous Dynamical Systems-A}, 21(1):41, 2008.
	
	\bibitem{br}
	H.~Berestycki and L.~Rossi.
	\newblock Generalizations and properties of the principal eigenvalue of
	elliptic operators in unbounded domains.
	\newblock {\em Communications on Pure and Applied Mathematics},
	68(6):1014--1065, 2015.
	
	\bibitem{evans}
	L.~C. Evans.
	\newblock Partial differential equations. graduate studies in mathematics.
	\newblock {\em American mathematical society}, 2:1998, 1998.
	
	\bibitem{fife}
	P.~C. Fife and J.~B. McLeod.
	\newblock The approach of solutions of nonlinear diffusion equations to
	travelling front solutions.
	\newblock {\em Archive for Rational Mechanics and Analysis}, 65(4):335--361,
	1977.
	
	\bibitem{fisher}
	R.~A. Fisher.
	\newblock The wave of advance of advantageous genes.
	\newblock {\em Annals of eugenics}, 7(4):355--369, 1937.
	
	\bibitem{gatto2020spread}
	M.~Gatto, E.~Bertuzzo, L.~Mari, S.~Miccoli, L.~Carraro, R.~Casagrandi, and
	A.~Rinaldo.
	\newblock Spread and dynamics of the covid-19 epidemic in italy: Effects of
	emergency containment measures.
	\newblock {\em Proceedings of the National Academy of Sciences},
	117(19):10484--10491, 2020.
	
	\bibitem{GT}
	D.~Gilbarg and N.~S. Trudinger.
	\newblock {\em Elliptic partial differential equations of second order}.
	\newblock springer, 2015.
	
	\bibitem{giletti2015kpp}
	T.~Giletti, L.~Monsaingeon, and M.~Zhou.
	\newblock A kpp road--field system with spatially periodic exchange terms.
	\newblock {\em Nonlinear Analysis}, 128:273--302, 2015.
	
	\bibitem{kinezaki2003modeling}
	N.~Kinezaki, K.~Kawasaki, F.~Takasu, and N.~Shigesada.
	\newblock Modeling biological invasions into periodically fragmented
	environments.
	\newblock {\em Theoretical population biology}, 64(3):291--302, 2003.
	
	\bibitem{KPP}
	A.~Kolmogorov, I.~Petrovsky, and N.~Piskunov.
	\newblock Investigation of the equation of diffusion combined with increasing
	of the substance and its application to a biology problem.
	\newblock {\em Bull. Moscow State Univ. Ser. A: Math. Mech}, 1(6):1--25, 1937.
	
	\bibitem{ladyzhenskaya}
	O.~A. Ladyzhenskaia, V.~A. Solonnikov, and N.~N. Ural'tseva.
	\newblock {\em Linear and quasi-linear equations of parabolic type}, volume~23.
	\newblock American Mathematical Soc., 1968.
	
	\bibitem{okubo1980diffusion}
	A.~Okubo et~al.
	\newblock Diffusion and ecological problems: mathematical models.
	\newblock 1980.
	
	\bibitem{pauthier2015uniform}
	A.~Pauthier.
	\newblock Uniform dynamics for fisher-kpp propagation driven by a line of fast
	diffusion under a singular limit.
	\newblock {\em Nonlinearity}, 28(11):3891, 2015.
	
	\bibitem{pauthier2016influence}
	A.~Pauthier.
	\newblock The influence of nonlocal exchange terms on fisher--kpp propagation
	driven by a line of fast diffusion.
	\newblock {\em Communications in Mathematical Sciences}, 14(2):535--570, 2016.
	
	\bibitem{potapov2004climate}
	A.~B. Potapov and M.~A. Lewis.
	\newblock Climate and competition: the effect of moving range boundaries on
	habitat invasibility.
	\newblock {\em Bulletin of Mathematical Biology}, 66(5):975--1008, 2004.
	
	\bibitem{robinet2012human}
	C.~Robinet, C.-E. Imbert, J.~Rousselet, D.~Sauvard, J.~Garcia, F.~Goussard, and
	A.~Roques.
	\newblock Human-mediated long-distance jumps of the pine processionary moth in
	europe.
	\newblock {\em Biological invasions}, 14(8):1557--1569, 2012.
	
	\bibitem{rossi2017effect}
	L.~Rossi, A.~Tellini, and E.~Valdinoci.
	\newblock The effect on fisher-kpp propagation in a cylinder with fast
	diffusion on the boundary.
	\newblock {\em SIAM Journal on Mathematical Analysis}, 49(6):4595--4624, 2017.
	
	\bibitem{shigesada1997biological}
	N.~Shigesada and K.~Kawasaki.
	\newblock {\em Biological invasions: theory and practice}.
	\newblock Oxford University Press, UK, 1997.
	
	\bibitem{shigesada1986traveling}
	N.~Shigesada, K.~Kawasaki, and E.~Teramoto.
	\newblock Traveling periodic waves in heterogeneous environments.
	\newblock {\em Theoretical Population Biology}, 30(1):143--160, 1986.
	
	\bibitem{skellam}
	J.~G. Skellam.
	\newblock Random dispersal in theoretical populations.
	\newblock {\em Biometrika}, 38(1/2):196--218, 1951.
	
	\bibitem{tellini2019comparison}
	A.~Tellini.
	\newblock Comparison among several planar fisher-kpp road-field systems.
	\newblock In {\em Contemporary Research in Elliptic PDEs and Related Topics},
	pages 481--500. Springer, 2019.
	
\end{thebibliography}
\end{document}